\documentclass[11pt]{amsart}
\usepackage{amssymb,amsmath,color,hyperref}
\usepackage[mathscr]{eucal}

\theoremstyle{plain}
\newtheorem{thm}{Theorem}[section]

\newtheorem{lem}[thm]{Lemma}

\newtheorem{prop}[thm]{Proposition}

\theoremstyle{definition}

\newtheorem{ex}[thm]{Example}

\theoremstyle{remark}
\newtheorem{rem}[thm]{Remark}

\newtheorem*{remark*}{Remark}

\numberwithin{equation}{section}

        \newcommand{\field}[1]{{\mathbb{#1}}}
        \newcommand{\NN}{\field{N}}
        \newcommand{\ZZ}{\field{Z}}
        
        \newcommand{\RR}{\field{R}}
        \newcommand{\CC}{\field{C}}

\allowdisplaybreaks

\begin{document}

\title[Bochner-Schr\"odinger operator on symplectic manifolds]{Semiclassical spectral analysis of the Bochner-Schr\"odinger operator on symplectic manifolds of bounded geometry}

\author[Y. A. Kordyukov]{Yuri A. Kordyukov}
\address{Institute of Mathematics, Ufa Federal Research Centre, Russian Academy of Sciences, 112~Chernyshevsky str., 450008 Ufa, Russia and Kazan Federal University, 18 Kremlyovskaya str., Kazan, 420008, Russia} \email{yurikor@matem.anrb.ru}

\thanks{Partially supported by the development program of the Regional Scientific and Educational Mathematical Center of the Volga Federal District, agreement N 075-02-2020-1478.}

\subjclass[2000]{Primary 58J37; Secondary 53D50}

\keywords{symplectic manifold, Bochner-Schr\"odinger operator, semiclassial asymptotics, spectrum,  manifolds of bounded geometry}

\begin{abstract}
We study the Bochner-Schr\"odinger operator $H_{p}=\frac 1p\Delta^{L^p\otimes E}+V$ on high tensor powers of a positive line bundle $L$ on a symplectic manifold of bounded geometry. First, we give a rough asymptotic description of its spectrum in terms of the spectra of the model operators. This allows us to prove the existence of gaps in the spectrum under some conditions on the curvature of the line bundle. Then we consider the spectral projection of such an operator corresponding to an interval whose extreme points are not in the spectrum and study asymptotic behavior of its kernel. First, we establish the off-diagonal exponential estimate. Then we state a complete asymptotic expansion in a fixed neighborhood of the diagonal.
\end{abstract}

%\begin{classification}
%Primary 58J37; Secondary 53D50
%\end{classification}

%\begin{keywords}
% symplectic manifold, Bochner-Schr\"odinger operator, semiclassial asymptotics, spectrum,  manifolds of bounded geometry
%\end{keywords}

\date{December 28, 2020}

 \maketitle
%\tableofcontents
\section{Introduction}
\subsection{The setting}
Let $(X,g)$ be a smooth Riemannian manifold of dimension $d$ without boundary, $(L,h^L)$ a Hermitian line bundle on $X$ with a Hermitian connection $\nabla^L$ and $(E,h^E)$ a Hermitian vector bundle of rank $r$ on $X$ with a Hermitian connection $\nabla^E$. We will suppose that $(X, g)$ is  a manifold of bounded geometry and $L$ and $E$ have bounded geometry. This means that the curvatures $R^{TX}$, $R^L$ and $R^E$ of the Levi-Civita connection $\nabla^{TX}$, connections $\nabla^L$ and $\nabla^E$, respectively, and their derivatives of any order are uniformly bounded on $X$ in the norm induced by $g$, $h^L$ and $h^E$, and the injectivity radius $r_X$ of $(X, g)$ is positive.

For any $p\in \NN$, let $L^p:=L^{\otimes p}$ be the $p$th tensor power of $L$ and let
\[
\nabla^{L^p\otimes E}: {C}^\infty(X,L^p\otimes E)\to
{C}^\infty(X, T^*X \otimes L^p\otimes E)
\] 
be the Hermitian connection on $L^p\otimes E$ induced by $\nabla^{L}$ and $\nabla^E$. Consider the induced Bochner Laplacian $\Delta^{L^p\otimes E}$ acting on $C^\infty(X,L^p\otimes E)$ by
\begin{equation}\label{e:def-Bochner}
\Delta^{L^p\otimes E}=\big(\nabla^{L^p\otimes E}\big)^{\!*}\,
\nabla^{L^p\otimes E},
\end{equation} 
where $\big(\nabla^{L^p\otimes E}\big)^{\!*}: {C}^\infty(X,T^*X\otimes L^p\otimes E)\to
{C}^\infty(X,L^p\otimes E)$ is the formal adjoint of  $\nabla^{L^p\otimes E}$. Let $V\in C^\infty(X,\operatorname{End}(E))$ be a self-adjoint endomorphism of $E$. We will assume that $V$ and its derivatives of any order are uniformly bounded on $X$ in the norm induced by $g$ and $h^E$. 
We will study the Bochner-Schr\"odinger operator $H_p$ acting on $C^\infty(X,L^p\otimes E)$ by
\[
H_{p}=\frac 1p\Delta^{L^p\otimes E}+V. 
\] 
Since $(X,g)$ is complete, the operator $H_p$ is essentially self-adjoint  in the Hilbert space $L^2(X,L^p\otimes E)$ with initial domain  $C^\infty_c(X,L^p\otimes E)$, see \cite[Theorem 2.4]{ko-ma-ma}. We still denote by $H_p$ its unique self-adjoint extension, and by $\sigma(H_p)$ its spectrum in $L^2(X,L^p\otimes E)$.

Consider the real-valued closed 2-form $\mathbf B$ (the magnetic field) given by 
\begin{equation}\label{e:def-omega}
\mathbf B=iR^L. 
\end{equation} 
%where $R^L$ is the curvature of the connection $\nabla^L $ defined as $R^L=(\nabla^L)^2$. 
We assume that $\mathbf B$ is non-degenerate. Thus, $X$ is a symplectic manifold. In particular, its dimension is even, $d=2n$, $n\in \NN$. 

For $x\in X$, let $B_x : T_xX\to T_xX$ be the skew-adjoint operator such that 
\[
\mathbf B_x(u,v)=g(B_xu,v), \quad u,v\in T_xX. 
\]
The operator $|B_x|:=(B_x^*B_x)^{1/2} : T_xX\to T_xX$ is a positive self-adjoint operator. We assume that it is uniformly positive on $X$: 
\begin{equation}\label{e:uniform-positive}
|B_x|\geq b_0>0, \quad x\in X.
\end{equation} 

\begin{rem}
The operator $H_p$ was introduced and studied by Demailly in \cite{Demailly85,Demailly91}. The study of its spectrum plays an important role in the proof of holomorphic Morse inequalities. 
\end{rem}

\begin{rem}
 Assume that the Hermitian line bundle $(L,h^L)$ is trivial and $(E,h^E)$ is a trivial Hermitian line bundle with a trivial connection $\nabla^E$. Then we can write $\nabla^L=d-i \mathbf A$ with a real-valued 1-form $\mathbf A$ (the magnetic potential), and we have
\[
R^L=-id\mathbf A,\quad \mathbf B=d\mathbf A. 
\]
The operator $H_p$ is related with the semiclassical magnetic Schr\"odinger operator
\[
H_p=\hbar^{-1}[(i\hbar d+\mathbf A)^*(i\hbar d+\mathbf A)+\hbar V], \quad \hbar=\frac{1}{p},\quad p\in \NN. 
\]
It can be also considered as the magnetic Schr\"odinger operator with strong electric and magnetic fields, growing at the same rate: 
\[
H_p=\frac{1}{p}[(d-ip\mathbf A)^*(d-ip\mathbf A)+pV], \quad p\in \NN. 
\]
\end{rem} 

\begin{rem} 
If $X$ is the Euclidean space $\RR^{2n}$ with coordinates $Z=(Z_1,\ldots,Z_{2n})$, we can write the 1-form $\bf A$ as
\[
{\bf A}= \sum_{j=1}^{2n}A_j(Z)\,dZ_j,
\]
the matrix of the Riemannian metric $g$ as $g(Z)=(g_{j\ell}(Z))_{1\leq j,\ell\leq 2n}$
and its inverse as $g(Z)^{-1}=(g^{j\ell}(Z))_{1\leq j,\ell\leq 2n}$.
Denote $|g(Z)|=\det(g(Z))$. Then $\bf B$ is given by 
\[
{\bf B}=\sum_{j<k}B_{jk}\,dZ_j\wedge dZ_k, \quad
B_{jk}=\frac{\partial A_k}{\partial Z_j}-\frac{\partial
A_j}{\partial Z_k}.
\]
Moreover, the operator $H_p$ has the form
\[
H_p=\frac{1}{p}\frac{1}{\sqrt{|g|}}\sum_{1\leq j,\ell\leq 2n}\left(i \frac{\partial}{\partial Z_j}+pA_j\right) \left[\sqrt{|g|} g^{j\ell} \left(i \frac{\partial}{\partial Z_\ell}+pA_\ell\right)\right]+V.
\]
Our assumptions hold, if the matrix $(B_{j\ell}(Z))$ has full rank $2n$ and its eigenvalues are separated from zero uniformly on $Z\in \RR^{2n}$, for any $\alpha \in \ZZ^{2n}_+$ and $1\leq j,\ell\leq 2n$, we have 
\[
\sup_{Z\in \RR^{2n}}|\partial^\alpha g_{j\ell}(Z)|<\infty, \quad \sup_{Z\in \RR^{2n}}|\partial^\alpha B_{j\ell}(Z)|<\infty, 
\]
and the matrix $(g_{j\ell}(Z))$ is positive definite uniformly on $Z\in \RR^{2n}$.
 \end{rem}

\subsection{Description of the spectrum}
Our first result gives an asymptotic description of the spectrum of $H_{p}$ as $p\to \infty$ in terms of the spectra of the model operators.

For an arbitrary $x_0\in X$, the model operator at $x_0$ is a second order differential operator $\mathcal H^{(x_0)}_{p}$, acting on $C^\infty(T_{x_0}X, E_{x_0})$, which is obtained from the operator $H_p$ by freezing coefficients at $x_0$. This operator was introduced by Demailly \cite{Demailly85,Demailly91}.

Consider the trivial Hermitian line bundle $L_0$ over $T_{x_0}X$ and the trivial Hermitian vector bundle $E_0$ over $T_{x_0}X$ with the fiber $E_{x_0}$. We introduce the connection 
\begin{equation}\label{e:nablaL0}
\nabla^{L^p_0}_{v}=\nabla_{v}-ip\alpha_v, 
\end{equation}
acting on $C^\infty(T_{x_0}X, L^p_0\otimes E_0)\cong C^\infty(T_{x_0}X, E_{x_0})$, with the connection one-form $\alpha\in C^\infty(T(T_{x_0}X),\RR)$ given by 
\begin{equation}\label{e:Aflat}
\alpha_v(w)=\frac{1}{2}\mathbf B_{x_0}(v,w),\quad v,w\in T_{x_0}X. 
\end{equation}
Its curvature is constant: $d\alpha=\mathbf B_{x_0}$. 
Denote by $\Delta^{L_0^p}$ the associated Bochner Laplacian.
The model operator $\mathcal H^{(x_0)}_{p}$ on $C^\infty(T_{x_0}X, E_{x_0})$ is defined as 
\begin{equation}\label{e:DeltaL0p}
\mathcal H^{(x_0)}_{p}=\frac 1p\Delta^{L_0^p}+V(x_0).
\end{equation}

Since $B_{x_0}$ is skew-adjoint, its eigenvalues have the form $\pm i a_j(x_0), j=1,\ldots,n,$ with $a_j(x_0)>0$. By \eqref{e:uniform-positive}, $a_j(x_0)\geq b_0>0$ for any $x_0\in X$ and $j=1,\ldots,n$.  Denote by $V_\mu(x_0), \mu=1,\ldots,r$, the eigenvalues of $V(x_0)$. The spectrum of $\mathcal H^{(x_0)}_{p}$ is independent of $p$ and consists of eigenvalues of infinite multiplicity 
\begin{equation}\label{e:def-Sigmax}
\sigma(\mathcal H^{(x_0)}_{p})=\Sigma_{x_0}:=\left\{\Lambda_{\mathbf k,\mu}({x_0})\,:\, \mathbf k\in\ZZ_+^n, \mu=1,\ldots,r\right\}, 
\end{equation}
where, for $\mathbf k=(k_1,\cdots,k_n)\in\ZZ_+^n$, $\mu=1,\ldots,r$ and $x_0\in X$,
\begin{equation}\label{e:def-Lambda}
\Lambda_{\mathbf k,\mu}(x_0)=\sum_{j=1}^n(2k_j+1) a_j(x_0)+V_\mu(x_0).
\end{equation}
In particular, the lowest eigenvalue of $\mathcal H^{(x_0)}_{p}$ is 
\[
\Lambda_0(x_0):=\sum_{j=1}^n a_j(x_0)+\min _\mu V_\mu(x_0). 
\]
Let $\Sigma$ be the union of the spectra of the model operators: 
\[ 
\Sigma=\bigcup_{x\in X}\Sigma_x=\left\{\Lambda_\mathbf {k,\mu}(x)\,:\, \mathbf k\in\ZZ_+^n, \mu=1,\ldots,r, x\in X \right\}.
\]

\begin{thm}\label{t:spectrum}
For any $K>0$, there exists $c>0$ such that for any $p\in \NN$ the spectrum of $H_{p}$ in the interval  $[0,K]$  is  contained in the $cp^{-1/4}$-neighborhood of $\Sigma$.  
\end{thm}

The set $\Sigma$ is a closed subset of $\RR$, which can be represented as the union of the closed intervals (bands):
\[
\Sigma=\bigcup_{\mathbf k\in\ZZ_+^n, \mu=1,\ldots,r}[\alpha_{\mathbf k,\mu}, \beta_{\mathbf k,\mu}]
\]
where, for any $\mathbf k\in\ZZ_+^n$ and $\mu=1,\ldots,r$, the interval $[\alpha_{\mathbf k,\mu}, \beta_{\mathbf k,\mu}]$ is the range of the function $\Lambda_{\mathbf k,\mu}$ on $X$: $[\alpha_{\mathbf k,\mu}, \beta_{\mathbf k,\mu}]=\{\Lambda_{\mathbf k,\mu}({x_0}) : x_0\in X\}$.

In general, the bands $[\alpha_{\mathbf k,\mu}, \beta_{\mathbf k,\mu}]$ can overlap without any gaps and $\Sigma$ is the semi-axis $[\Lambda_0,+\infty)$ with $\Lambda_0=\inf_{x\in X} \Lambda_0(x)$. In this case, Theorem \ref{t:spectrum} tells nothing about the location of the spectrum of $H_{p}$, except for the lower bound for its bottom $\lambda_0(H_p)=\inf \sigma(H_p)$: 
\begin{equation}\label{e:lower-lambda0}
\lambda_0(H_p)\geq \Lambda_0-cp^{-1/4}, \quad p\in \NN.
\end{equation}
This estimate agrees with the similar estimate for the magnetic Laplacian obtained in \cite[Theorem 3.1]{HM96} (see also \cite{Morin19}). One should also emphasize that we make no assumptions on the curvature $\mathbf B$ except for full-rank condition. There is a number of papers devoted to the study of the asymptotic behavior of low-lying eigenvalues of the magnetic Schr\"odinger operator under some additional assumptions on the magnetic field like the existence of non-degenerate magnetic wells (see \cite{HK14,Morin19,raymond-book} and references therein for the case of non-degenerate magnetic field).

In some cases, $\Sigma$ has gaps: $[\Lambda_0,+\infty)\setminus \Sigma\not=\emptyset$. For instance, if $V(x)\equiv 0$ and the functions $a_j$ can be chosen to be constants:
\begin{equation}\label{e:aj-constant}
a_j(x)\equiv a_j, \quad x\in X, \quad  j=1,\ldots,n,
\end{equation}
 then $\Sigma$ is a countable discrete set.
In particular, if $J=\frac{1}{2\pi}B$ is an almost-complex structure (the almost K\"ahler case) and $V(x)\equiv 0$,, then $a_j=2\pi,  j=1,\ldots,n$ and
\begin{equation}\label{e:Kaehler}
\Sigma=\left\{2\pi (2k+n)\,:\, k\in\ZZ_+\right\}.
\end{equation}
The set $\Sigma$ may also have gaps if the functions $a_j$ are not  constants, but varies slow enough. In these cases, Theorem \ref{t:spectrum} implies the existence of gaps in  the spectrum of $H_{p}$. In particular, when $V(x)\equiv 0$ and condition \eqref{e:aj-constant} holds, then the spectrum of $H_{p}$ is contained in the union of neighborhoods of $a_j$'s of size $O(p^{-1/4})$. In the almost-K\"ahler case, Theorem \ref{t:spectrum} was proved in \cite{FT}. Our proof uses some ideas and constructions of the proof given in \cite{FT}, but it has some improvements and is shorter. 

The following example demonstrates what kind of information about the eigenvalues of $H_p$ can be obtained from Theorem \ref{t:spectrum} in the alomst-K\"ahler case.   

\begin{ex}
Suppose that $X$ is the unit two-sphere $S^2=\{(x,y,z)\in \RR^3:x^2+y^2+z^2=1\}$. 
In the spherical coordinates $x=\sin\theta \cos\varphi, y=\sin\theta \sin\varphi, z=\cos\theta,\theta\in (0,\pi), \varphi\in (0,2\pi)$, we take the Riemannian metric $g=R^2(d\theta^2+\sin^2\theta d\varphi^2)$, and the magnetic field $\mathbf B=\frac{1}{2}\sin\theta d\theta\wedge d\varphi$. Let $L$ be the corresponding quantum line bundle. 
%associated with the Hopf bundle $S^3\to S^2$, and the character $\chi : S^1\to S^1$, $\chi(u)=u, u\in S^1$.  
The only eigenvalue $a_1(\theta, \varphi)$ can be found from the relation $\mathbf B=a_1dv_X$, which gives
\[ 
a_1(\theta, \varphi)=\frac{1}{2R^2}, \quad \Sigma=\left\{ (2k+1)\frac{1}{2R^2}\,:\, k\in\ZZ_+ \right\}.
\]
By the classical formula for the eigenvalues of the magnetic Laplacian $\Delta^{L^p}$  [Tamm (1931), Wu-Yang (1976)], the eigenvalues of $H_p=\frac 1p\Delta^{L^p}$ are given by
\[
\nu_{p,k}=(2k+1)\frac{1}{2R^2}+\frac{k(k+1)}{R^2p}, \quad  k\in\ZZ_+
\]
with multiplicity $m_{p,k}=p+2k+1$.
In particular, if the metric $g$ is K\"ahler, then $\frac{1}{2\pi}\mathbf B=dv_X$, which gives $R^2=\frac{1}{4\pi}$ and (cf. \eqref{e:Kaehler}) 
\[
\nu_{p,k}=2\pi (2k+1)+4\pi k(k+1)\frac 1p, \quad  k\in\ZZ_+.
\]
So we see that, in this example, Theorem \ref{t:spectrum} describes the leading term in the asymptotic expansion of each eigenvalue. Note that a description of the leading term in the asymptotic expansion of the multiplicites is given by Demailly's theorem (see \eqref{e:Demailly} below).
\end{ex}

Another way to obtain an operator $H_p$ with a gap in $\Sigma$ is to take 
$V(x)=-\tau(x)$ with $\tau(x):=\sum_{j=1}^n a_j(x), x\in X$. Then $H_p=\frac{1}{p}\Delta_p$, where  $\Delta_p:=\Delta^{L^p\otimes E}-p\tau$ is the renormalized Bochner Laplacian introduced by Guillemin and Uribe in \cite{Gu-Uribe}.  We get $\Lambda_0(x)\equiv 0$ and $\Sigma$ has a gap around zero: $\Sigma=\{0\}\cup [2b_0, \infty)$ with $b_0=\inf_{x\in X}|B_x|>0$. In this case, a better estimate for the spectrum of $H_p$ (with $p^{-1}$ instead of $p^{-1/4}$) holds true: there exists $c>0$ such that for any $p\in \NN$ the spectrum of $H_{p}$  is  contained in $(-cp^{-1}, cp^{-1})\cup [2b_0-cp^{-1}, \infty)$. This  estimate (with not precised constant $h_0$) is proved in \cite{Gu-Uribe} when $X$ is compact and $E$ is a  trivial line bundle. It was proved for a general vector bundle $E$ on a compact manifold in \cite[Corollary 1.2]{ma-ma02} and for manifolds of bounded geometry in \cite[Theorem 1.1]{ko-ma-ma} (see also the references therein for some related works). In particular, the estimate \eqref{e:lower-lambda0} holds with $p^{-1}$ instead of $p^{-1/4}$ (cf. \cite[Remark 2.3]{HM96} in the case of the magnetic Laplacian).

By constructing approximate eigenfunctions, one can show that each $\Lambda\in \Sigma$ is close to the spectrum of $H_p$ (cf. \cite[Theorem 2.2]{HM96}  in the case of the magnetic Laplacian). Actually, when $X$ os compact, any neighborhood of $\Lambda$ contains infinitely many eigenvalues as follows from the asymptotic formula for the eigenvalue distribution function of the operator  $H_p$ proved by Demailly \cite{Demailly85,Demailly91}. Let us briefly recall this result.

Suppose that $X$ is compact, The eigenvalue distribution function $N_p(\lambda)$ of $H_p$ is defined by
\[
N_p(\lambda)=\#\{j\in \ZZ_+: \nu_{p,j}\leq \lambda \},\quad \lambda\in \RR, 
\]
where $\nu_{p,j}, j\in \ZZ_+$ are the eigenvalues of $H_p$ taken with multiplicities. For any $x\in X$, let $N(x,\lambda)$ be the eigenvalue distribution function of the model operator $\mathcal H^{(x)}_{p}$ defined by
\[
N(x,\lambda)=\#\{(\mathbf k, \mu) : \Lambda_{\mathbf k,\mu}(x)\leq \lambda \},\quad \lambda\in \RR.
\]
By \cite[Theorem 0.6]{Demailly85} (see also \cite[Corollary 3.3]{Demailly91}), there exists a countable set $\mathcal D\subset \RR$ such that for any $\lambda\in \RR\setminus \mathcal D$  
\begin{equation}\label{e:Demailly1}
\lim_{p\to +\infty}p^{-n}N_p(\lambda)=\frac{1}{(2\pi)^n} \int_X \left(\prod_{j=1}^n a_j(x)\right) N(x,\lambda) dv_X(x),
\end{equation}
where $dv_X$ denotes the Riemannian volume form. The formula \eqref{e:Demailly1} can be rewritten in terms of the Liouville volume form $\Omega_{\mathbf B}=\frac{1}{n!} \mathbf B^n$ as follows:
\begin{equation}\label{e:Demailly2}
\lim_{p\to +\infty}p^{-n}N_p(\lambda)=\frac{1}{(2\pi)^n} \int_X N(x,\lambda)\Omega_{\mathbf B}(x).  
\end{equation}
By \eqref{e:Demailly2}, for any interval $(\alpha,\beta)$, we have
\begin{multline}\label{e:Demailly}
\#\{j\in \ZZ_+: \nu_{p,j}\in (\alpha,\beta)\}\\ =\frac{p^{n}}{(2\pi)^n} \sum_{\mathbf k, \mu}\mu_{\mathbf B}(\{x\in X : \Lambda_{\mathbf k,\mu}(x)\in (\alpha,\beta)\})+o(p^{n}), \quad  p\to \infty,
\end{multline}
where $\mu_{\mathbf B}$ denotes the Liouville measure. In particular, if $(\alpha,\beta)\cap \Sigma=\emptyset$, 
\[
\#\{j\in \ZZ_+: \nu_{p,j}\in (\alpha,\beta)\}=o(p^{n}), \quad p\to \infty.  
\]

Apparently, Theorem \ref{t:spectrum} holds in the case when the magnetic field degenerates. If $\mathbf B_{x_0}$ is degenerate for some $x_0\in X$, the model operator $\mathcal H^{(x_0)}_{p}$ is still well defined, but its spectrum is the semi-axis $[V(x_0),\infty)$. Then again, Theorem \ref{t:spectrum} contains  no information about the location of the spectrum of $H_{p}$, except the lower bound for its bottom. Therefore, we restrict ourselves to the case when $\mathbf B$ is non-degenerate.

\subsection{Asymptotic behavior of the spectral projection}
Consider an interval $I=(\alpha,\beta)$ such that $\alpha,\beta\not \in \Sigma$. 
By Theorem \ref{t:spectrum}, there exists $\mu_0>0$ and $p_0\in \NN$ such that for any $p>p_0$ 
\[
\sigma(H_{p})\subset (-\infty, \alpha-\mu_0) \cup I \cup (\beta+\mu_0, \infty).
\] 
Let $P_{p,I}$ be the spectral projection of the operator $H_{p}$ associated with $I$  and $P_{p,I}(x,x^\prime)$, $x,x^\prime\in X$, be its smooth kernel with respect to the Riemannian volume form $dv_X$. We study the asymptotic behavior of the kernel $P_{p,I}(x,x^\prime)$ as $p\to \infty$.

First, we establish the off-diagonal exponential estimate for $P_{p,I}(x,x^\prime)$.

%\begin{thm}\label{t:far-off-diagonal}
%There exists $c>0$ such that, for any $k\in \mathbb N$ and $\varepsilon>0$, there exists $C>0$ such that for any $p>p_0$,  
%\begin{equation}\label{e:far-off-diagonal}
%|P_{p,I}(x,x^\prime)|_{C^k}\leq Ce^{-c\sqrt{p}}, \quad d(x,x^\prime)>\varepsilon.
%\end{equation}  
%\end{thm}

\begin{thm}\label{t:exp-Pp}
There exists $c>0$ such that for any $k\in \mathbb N$, there exists $C_k>0$ such that for any $p\in \mathbb N$, $x, x^\prime \in X$, we have
\begin{equation}\label{e1.9}
\big|P_{p,I}(x, x^\prime)\big|_{{C}^k}\leq C_k p^{n+\frac{k}{2}}
e^{-c\sqrt{p} \,d(x, x^\prime)}.
\end{equation}
\end{thm}

Here $d(x,x^\prime)$ is the geodesic distance and $|P_{p,I}(x, x^\prime)|_{{C}^k}$ denotes the pointwise ${C}^k$-seminorm of the section $P_{p,I}$ at a point $(x, x^\prime)\in X\times X$, which is the sum of the norms induced by $h^L, h^E$ and $g$ of the derivatives up to order $k$ of $P_{p,I}$ with respect to the connection $\nabla^{L^p\otimes E}$ and the Levi-Civita connection $\nabla^{TX}$ evaluated at $(x, x^\prime)$.

Then we describe an asymptotic expansion of  $P_{p,I}$ as $p\to \infty$ in a fixed neighborhood of the diagonal (independent of $p$). Such kind of expansion is called full off-diagonal expansion.

First, we introduce normal coordinates near an arbitrary point $x_0\in X$. 
We denote by $B^{X}(x_0,r)$ and $B^{T_{x_0}X}(0,r)$ the open balls in $X$ and $T_{x_0}X$ with center $x_0$ and radius $r$, respectively. We identify $B^{T_{x_0}X}(0,r_X)$ with $B^{X}(x_0,r_X)$ via the exponential map $\exp^X_{x_0}$. Furthermore, we choose trivializations of the bundles $L$ and $E$ over $B^{X}(x_0,r_X)$,   identifying their fibers $L_Z$ and $E_Z$ at $Z\in B^{T_{x_0}X}(0,r_X)\cong B^{X}(x_0,r_X)$ with the spaces  $L_{x_0}$ and $E_{x_0}$ by parallel transport with respect to the connections $\nabla^L$ and $\nabla^E$ along the curve $\gamma_Z : [0,1]\ni u \to \exp^X_{x_0}(uZ)$. Denote by $\nabla^{L^p\otimes E}$ and $h^{L^p\otimes E}$ the connection and the Hermitian metric on the trivial bundle with fiber $(L^p\otimes E)_{x_0}$ induced by these trivializations.  

We choose an orthonormal base $\{e_j : j=1,\ldots,2n\}$ in $T_{x_0}X$ such that  
\begin{equation}\label{e:obase}
B_{x_0}e_{2k-1}=a_k(x_0)e_{2k}, \quad B_{x_0}e_{2k}=-a_k(x_0)e_{2k-1},\quad k=1,\ldots,n. 
\end{equation}
It gives rise to a coordinate chart $\gamma_{x_0} : B(0,c)\subset  \RR^{2n}\to X$  defined on the open ball $B(0,c)$ in $\RR^{2n}$ with center at the origin and radius $c\in (0,r_X)$, which is given by the restriction of the exponential map $\exp_{x_0}^X : T_{x_0}X \to X$ composed with the linear isomorphism $\mathbb R^{2n}\to T_{x_0}X$ determined by the base $\{e_j \}$.

Let $dv_{TX}$ denote the Riemannian volume form of the Euclidean space $(T_{x_0}X, g_{x_0})$. We define a smooth function $\kappa$ on $B^{T_{x_0}X}(0,r_X)\cong B^{X}(x_0,r_X)$ by the equation
\[
dv_{X}(Z)=\kappa(Z)dv_{TX}(Z), \quad Z\in B^{T_{x_0}X}(0,r_X). 
\] 
Consider the fiberwise product $TX\times_X TX=\{(Z,Z^\prime)\in T_{x_0}X\times T_{x_0}X : x_0\in X\}$. Let $\pi : TX\times_X TX\to X$ be the natural projection given by  $\pi(Z,Z^\prime)=x_0$. The kernel $P_{p,I}(x,x^\prime)$ induces a smooth section $P_{p,I,x_0}(Z,Z^\prime)$ of the vector bundle $\pi^*(\operatorname{End}(E))$ on $TX\times_X TX$ defined for all $x_0\in X$ and $Z,Z^\prime\in T_{x_0}X$ with $|Z|, |Z^\prime|<r_X$. 

\begin{thm}\label{t:main}
There exists $\varepsilon\in (0,r_X)$ such that for any $x_0\in X$ and $Z,Z^\prime\in T_{x_0}X$, $|Z|, |Z^\prime|<\varepsilon$, the sequence $P_{p,I,x_0}(Z,Z^\prime)$ admits an asymptotic expansion
\begin{equation}\label{e:main-expansion}
\frac{1}{p^n}P_{p,I,x_0}(Z,Z^\prime)\cong
\sum_{r=0}^\infty F_{r,x_0}(\sqrt{p} Z, \sqrt{p}Z^\prime)\kappa^{-\frac 12}(Z)\kappa^{-\frac 12}(Z^\prime)p^{-\frac{r}{2}}, 
\end{equation}
where the leading coefficient $F_{0,x_0}$ is the kernel of the spectral projection $\mathcal P_{I,x_0}$ of the model operator $\mathcal H^{(x_0)}:=\mathcal H^{(x_0)}_1$ associated with $I$:
\begin{equation}\label{e:F0}
F_{0,x_0}(Z,Z^\prime)=\mathcal P_{I, x_0}(Z,Z^\prime),
\end{equation}
and for any $r\geq 0$, the coefficients $F_{r,x_0}$ has the form
\begin{equation}\label{e:Fr}
F_{r,x_0}(Z,Z^\prime)=J_{r,x_0}(Z,Z^\prime)\mathcal P_{x_0}(Z,Z^\prime),
\end{equation}
where $\mathcal  P_{x_0}\in C^\infty(\RR^{2n}\times \RR^{2n})$ is the Bergman kernel  (see \eqref{e:Bergman} below) and $J_{r,x_0}(Z,Z^\prime)$ is a polynomial in $Z, Z^\prime$ with values in $\operatorname{End}(E_{x_0})$, depending smoothly on $x_0$, with the same parity as $r$ 
and $\operatorname{deg} J_{r,x_0}\leq \kappa(I)+3r$, where $\kappa(I)=\max \{|\mathbf k| : \Lambda_{\mathbf k,\mu}\in I\}$. 

For any $j\in \mathbb N$, the remainder 
\[
R_{j,p,x_0}(Z,Z^\prime):=\frac{1}{p^n}P_{p,I,x_0}(Z,Z^\prime)
-\sum_{r=0}^jF_{r,x_0}(\sqrt{p} Z, \sqrt{p}Z^\prime)\kappa^{-\frac 12}(Z)\kappa^{-\frac 12}(Z^\prime)p^{-\frac{r}{2}}
\]
in the asymptotic expansion \eqref{e:main-expansion} satisfies the following condition. For any $m,m^\prime\in \mathbb N$, there exist positive constants $C$, $c$, $c_0$ and $M$ such that for any $p\geq 1$, $x_0\in X$ and $Z,Z^\prime\in T_{x_0}X$, $|Z|, |Z^\prime|<\varepsilon$, 
\begin{multline}\label{e:main-exp}
\sup_{|\alpha|+|\alpha^\prime|\leq m}\Bigg|\frac{\partial^{|\alpha|+|\alpha^\prime|}}{\partial Z^\alpha\partial Z^{\prime\alpha^\prime}}R_{j,p,x_0}(Z,Z^\prime)\Bigg|_{C^{m^\prime}(X)}\\ 
\leq Cp^{-\frac{j-m+1}{2}}(1+\sqrt{p}|Z|+\sqrt{p}|Z^\prime|)^M\exp(-c\sqrt{p}|Z-Z^\prime|)+ O(e^{-c_0\sqrt{p}}).
\end{multline}
\end{thm}

Here $C^{m^\prime}(X)$ is the $C^{m^\prime}$-norm for the parameter $x_0\in X$. We say that $G_p=O(p^{-\infty})$ if for any $l, l_1\in \mathbb N$, there exists $C_{l,l_1}>0$ such that $C^{l_1}$-norm of $G_p$ is estimated from above by $C_{l,l_1}p^{-l}$. 

The spectral projection $\mathcal P_{I,x_0}$ can be written as
\begin{equation}
\label{e:Lambda-Bergman}
\mathcal P_{I,x_0}=\sum_{(\mathbf k,\mu) : \Lambda_{\mathbf k,\mu}\in I} \mathcal P_{\Lambda_{\mathbf k,\mu},x_0},
\end{equation}
where $\mathcal P_{\Lambda_{\mathbf k,\mu},x_0}$ is the orthogonal projection to the eigenspace of the model operator $\mathcal H^{(x_0)}$ with the eigenvalue $\Lambda_{\mathbf k,\mu}$. One can give an explicit formula for its smooth Schwartz kernel in terms of the Laguerre polynomials. For the lowest eigenvalue $\Lambda_0(x_0)$, the kernel of the projection $P_{\Lambda_0,x_0}$ has the form 
\[
\mathcal P_{0,x_0}(Z,Z^\prime)=\mathcal P_{x_0}(Z,Z^\prime)\pi_{0,x_0},
\]
where $\mathcal P_{x_0}\in C^\infty(\RR^{2n}\times \RR^{2n})$ is the Bergman kernel given by
\begin{equation}
\label{e:Bergman}
\mathcal P_{x_0}(Z,Z^\prime)=\frac{1}{(2\pi)^n}\prod_{j=1}^na_j \exp\left(-\frac 14\sum_{k=1}^na_k(|z_k|^2+|z_k^\prime|^2- 2z_k\bar z_k^\prime) \right)
\end{equation}
and $\pi_{0,x_0}$ is the orthogonal projection in $E_{x_0}$ to the eigenspace of $V(x_0)$ associated with its lowest eigenvalue.

In the case when $H_p=\frac{1}{p}\Delta_p$, where  $\Delta_p$ is the renormalized Bochner Laplacian and $I=(\alpha,\beta)$ is any interval, which contains $0$, with $\beta<2b_0$, the projection $P_{p,I}$ is called the generalized Bergman projection in \cite{ma-ma08}, since it generalizes the Bergman projection on complex manifolds. Its kernel is called the generalized Bergman kernel. 

Asymptotic expansions of the Bergman kernels on complex manifolds were studied for a long time and have many applications (see \cite{ma-ma:book} and also the references therein for the previous results).

For the Bergman kernel of the spin$^c$ Dirac operator on a symplectic manifold of bounded geometry, the same type of exponential estimate as in Theorems \ref{t:exp-Pp} is proved in \cite[Theorem 1]{ma-ma15}, and for the Bergman kernel of the renormalized Bochner Laplacian on a symplectic manifold of bounded geometry in \cite[Theorem 1.3]{ko-ma-ma}.

The full off-diagonal expansion for the Bergman kernel of the spin$^c$ Dirac operator was proved in \cite[Theorem 4.18']{dai-liu-ma} (see also \cite[Theorem 4.2.1]{ma-ma:book}). For the generalized Bergman kernel associated with the renormalized Bochner Laplacian, it was shown in \cite[Theorem 1.19]{ma-ma08} ((see also \cite[Theorem 4.1.24]{ma-ma:book})) that the off-diagonal expansion holds in a neighborhood of size $1/\sqrt{p}$ of the diagonal. This is called near off-diagonal expansion. In \cite{lu-ma-ma} a less precise estimate than in \cite[Theorem 1.19]{ma-ma08} was obtained in a neighborhood of size $p^{-\theta}$, $\theta\in(0,1/2)$. The proofs are based on the spectral gap property of the Bochner Laplacian, finite propagation speed arguments  for the wave equation and rescaling of the Bochner Laplacian near the diagonal, which is inspired by the analytic localization technique of Bismut-Lebeau \cite{BL}. In \cite{Kor18}, we combined the methods of \cite{dai-liu-ma,ma-ma:book,ma-ma08}  with weighted estimates with appropriate exponential weights as in \cite{Kor91,Meladze-Shubin1,Meladze-Shubin2} to prove the full off-diagonal expansion for the generalized Bergman kernel. In \cite{ko-ma-ma},  these results were extended to the case of manifolds of bounded geometry. Theorem~\ref{t:main} is a rather straightforward extension of the results of \cite{Kor18,ko-ma-ma}. In a companion paper \cite{Kor20a}, we apply the results of the paper to construct a Berezin-Toeplitz quantization associated with higher Landau levels of the Bochner Laplacian on a symplectic manifold. 
We mention that, in two simultaneous papers \cite{charles20a,charles20b}, Charles studies the same subject, using different methods. 

The paper is organized as follows. In Section~\ref{s:description}, we prove Theorem \ref{t:spectrum}. In Section \ref{s:res-estimates}, we prove weighted estimates for the resolvent of the operator $H_p$.  Section \ref{s:projection} is devoted to the study of the kernel of the spectral projections and contains the proofs of Theorems \ref{t:exp-Pp} and \ref{t:main}.

This work was started as a joint project with L. Charles, but later we decided to work on our approaches separately. I would like to thank Laurent for his collaboration. 

 \section{Description of the spectrum}\label{s:description}
This section is devoted to the proof of Theorem \ref{t:spectrum}. 

\subsection{Approximate inverse}
The proof of Theorem \ref{t:spectrum} is based on a construction of an approximate inverse for the operator $H_{p}-\lambda$ with $\lambda \not\in \Sigma$. The corresponding statement is given in the following proposition. 

 \begin{prop}\label{p:Klambda}
There exists a family $\{Q_p(\lambda): \lambda\not\in \Sigma\}$ of operators in $C^\infty_c(X, L^p\otimes E)$, which extend to bounded operators in $L^2(X, L^p\otimes E)$, such that, for any $\lambda\not\in \Sigma$, we have
\[
\left(H_{p}-\lambda\right)Q_p(\lambda)u=u+K_p(\lambda)u,\quad u\in C_c^\infty(X, L^p\otimes E),
\]
where $\{K_p(\lambda): \lambda\not\in \Sigma\}$ is a family of bounded operators in $L^2(X,L^p\otimes E)$, satisfying the following condition. For any $K>0$, there exists $C_K>0$ such that for any $\lambda\not\in \Sigma$, $|\lambda|<K$ and for any $p\in \NN$, we have
\[
\|K_p(\lambda):L^2(X,L^p\otimes E)\to L^2(X,L^p\otimes E)\| \leq C_K p^{-1/4}d(\lambda,\Sigma)^{-1},
\]
where $d(\lambda,\Sigma)$ denotes the distance from $\lambda$ to $\Sigma$. 
\end{prop}

Theorem \ref{t:spectrum} is an immediate consequence of Proposition \ref{p:Klambda}. To see this, let us fix $K>0$ and apply Proposition \ref{p:Klambda}. We get that, for any $\lambda\in\CC$ such that $d(\lambda,\Sigma)>cp^{-1/4}$ with $c=2C_K$ and $|\lambda|<K$, 
\[
\|K_p(\lambda):L^2(X,L^p\otimes E)\to L^2(X,L^p\otimes E)\| \leq \frac 12.
\]
Therefore, the operator $I +K_p(\lambda)$ is invertible in $L^2(X,L^p\otimes E)$. This immediately implies that the operator $H_{p}-\lambda$ is invertible in $L^2(X,L^p\otimes E)$ with
\[
\left(H_{p}-\lambda\right)^{-1}=Q_p(\lambda)(I +K_p(\lambda))^{-1},
\]
and $\lambda$ is not in $\sigma(H_{p})$ that completes the proof of Theorem \ref{t:spectrum}.

The proof of Proposition \ref{p:Klambda} will be given in the rest of this section. 

\subsection{Approximation by the model operator}
Our construction of an approximate inverse for the operator $H_{p}-\lambda$ with $\lambda \not\in \Sigma$ is based on the approximation of the operator $H_p$ by the model operator $\mathcal H^{(x_0)}_{p}$ in a sufficiently small neighborhood of an arbitrary point $x_0$. Since the spectrum of $\mathcal H^{(x_0)}_{p}$ coincides with $\Sigma_{x_0}$, the operator $\mathcal H^{(x_0)}_{p}-\lambda$ is invertible and we use its inverse to construct an approximate inverse for $H_{p}-\lambda$ in a small neighborhood of $x_0$. The global approximate inverse for $H_{p}-\lambda$ is constructed from these local approximate inverse, taking a cover by neighborhoods of size $O(p^{-1/4})$.

First, we construct some special coordinates near $x_0$  (see, for instance, \cite[Theorem 6.2.2]{FT}). 
We start with the coordinate chart $\gamma_{x_0} : B(0,c)\subset  \RR^{2n}\to X$ defined in Introduction. 
It satisfies 
\begin{equation}\label{e:x0gamma}
\gamma_{x_0}(0)=x_0,\quad (D\gamma_{x_0})_0 \left(\frac{\partial}{\partial Z_j}\right)=e_j,\quad j=1,\ldots, 2n,
\end{equation} 
and 
\begin{equation}\label{e:Bx0}
(\gamma_{x_0}^*\mathbf B)_{0}=\sum_{k=1}^n a_k(x_0)dZ_{2k-1}\wedge dZ_{2k}.
\end{equation}
Then, using Darboux Lemma, we deform $\gamma_{x_0}$ into a coordinate chart $\varkappa_{x_0} : B(0,c)\subset \RR^{2n}\to U_{x_0}=\varkappa_{x_0}(B(0,c))\subset X$ such that 
\begin{equation}\label{e:x0}
\varkappa_{x_0}(0)=x_0,\quad (D\varkappa_{x_0})_0 \left(\frac{\partial}{\partial Z_j}\right)=e_j,\quad j=1,\ldots, 2n,
\end{equation} 
and $\varkappa_{x_0}^*\mathbf B$ is a constant 2-form on $B(0,c)$ given by
\begin{equation}\label{e:kappaBx0}
(\varkappa_{x_0}^*\mathbf B)_Z=\sum_{k=1}^n a_k(x_0) dZ_{2k-1}\wedge dZ_{2k}\quad Z\in B(0,c).
\end{equation}
Moreover, as can be seen from the proof of the Darboux Lemma based on well-known Moser's argument (see, for instance, Lemma 3.14 in \cite[p.94]{McDuff-Salamon} and its proof), we can choose the coordinate charts $\varkappa_{x_0}$ so that, for every $k>0$, they are bounded in the $C^k$ norm uniformly with respect to $x_0$ in the sense that
\[
\|\varkappa^{-1}_{x_0} \circ \gamma_{x_0} : B(0,c)\to \RR^{2n}\|_{C^k}\leq C_k
\]
with $C_k$ independent of $x_0$.

%Using parallel transport along the rays through the origin, 
It is easy to see that there exists a trivialization of the Hermitian line bundle $L$ over $U_{x_0}$:
\[
\tau^L_{x_0} : U_{x_0}\times \CC \stackrel{\cong}{\to}L\left|_{U_{x_0}}\right.,
\]
such that the connection one-form of $\nabla^L$ in this trivialization coincides with 
the one-form $\alpha$ given by \eqref{e:Aflat}. We also assume that there exists a trivialization of the Hermitian bundle $E$ over $U_{x_0}$:
\[
\tau^E_{x_0} : U_{x_0}\times E_{x_0} \stackrel{\cong}{\to}E\left|_{U_{x_0}}\right.,
\]
These trivializations induce a trivialization of  $L^p\otimes E$ over $U_{x_0}$:
\[
\tau_{{x_0},p}=(\tau^L_{x_0})^p\otimes \tau^E_{x_0} : U_{x_0}\times E_{x_0} \stackrel{\cong}{\to}L^p\otimes E\left|_{U_{x_0}}\right..
\]
For any $x\in U_{x_0}$, we will write $\tau_{x_0,p}(x) : E_{x_0}\to L^p_x\otimes E_x$ for the associated linear map in the fibers. 

Let $g_{x_0}=\varkappa^*_{x_0} g$ be the Riemannian metric on $B(0,c)$ induced by the Riemannian metric $g$ on $X$. 
We introduce a map
\[
T^*_{{x_0},p} : C^\infty(X, L^p\otimes E)\to C^\infty(B(0,c), E_{x_0}),
\]
defined for $u\in C^\infty(X, L^p\otimes E)$ by 
\begin{equation}\label{e:defT}
T^*_{{x_0},p} u(Z)=|g_{x_0}(Z)|^{1/4}\tau_{{x_0},p}^{-1}(\varkappa_{x_0}(Z))[u(\varkappa_{x_0}(Z))],\quad  Z\in B(0,c).
\end{equation}
Consider the differential operator $H_p^{(x_0)}=T^*_{{x_0},p} \circ H_{p}\circ (T^*_{{x_0},p})^{-1}$
on $C^\infty(B(0,c), E_{x_0})$. It can be written as
\[
H_p^{(x_0)}=|g_{x_0}(Z)|^{1/4} \tau^*_{{x_0},p} \circ H_{p}\circ (\tau^*_{{x_0},p})^{-1} |g_{x_0}(Z)|^{-1/4}.
\]
Using the standard formula for the Bochner Laplacian in local coordinates, one can write
\begin{multline}\label{e:TDeltaT}
\tau^*_{{x_0},p} \circ H_{p}\circ (\tau^*_{{x_0},p})^{-1}\\ =-\frac 1p \sum_{\ell,m=1}^{2n}g_{x_0}^{\ell m}\nabla^{L^p\otimes E}_{e_\ell}\nabla^{L^p\otimes E}_{e_m}+\frac 1p \sum_{\ell=1}^{2n} \Gamma^\ell \nabla^{L^p\otimes E}_{e_\ell}+V_{x_0},
\end{multline}
where $\{e_j\}$ is the standard base in $\RR^{2n}$, $g_{x_0}^{\ell m}$ is the inverse of the matrix of $g_{x_0}$, $V_{x_0}=\tau^{E*}_{x_0} \circ V\circ (\tau^{E*}_{x_0})^{-1}\in C^\infty(B(0,c), \operatorname{End}(E_{x_0}))$ and $\Gamma^\ell\in C^\infty(B(0,c))$, $\ell=1,\ldots,2n,$ are some functions. If we denote by $\Gamma^E\in C^\infty(T(B(0,c)), \operatorname{End}(E_{x_0}))$ the connection one-form for the connection $\nabla^E$, we can write
\[
\nabla^{L^p\otimes E}_{v}=\nabla^{L^p_0}_{v}+\Gamma^E(v), \quad v\in T(B(0,c))=B(0,c)\times \RR^{2n}.
\]
where the connection $\nabla^{L^p_0}$ is given by \eqref{e:nablaL0}.

Then we have 
\[
|g_{x_0}|^{1/4} \nabla^{L^p\otimes E}_{v}|g_{x_0}|^{-1/4}=\nabla^{L^p_0}_{v}+\Gamma^E(v)-\frac 14v(\ln |g_{x_0}|).
\]
It follows that 
\begin{equation}\label{e:TDeltaT-D}
H_p^{(x_0)}=-\frac 1p \sum_{\ell,m=1}^{2n}g_{x_0}^{\ell m}\nabla^{L^p_0}_{e_\ell}\nabla^{L^p_0}_{e_m}+\frac 1p \sum_{\ell=1}^{2n} F_{\ell,{x_0}} \nabla^{L^p_0}_{e_\ell}+V_{x_0}+\frac 1pG_{x_0}
\end{equation}
with some $F_{\ell,{x_0}}, G_{x_0}\in C^\infty(B(0,c), \operatorname{End}(E_{x_0}))$, uniformly bounded on $x_0$. 

By \eqref{e:TDeltaT-D}, it follows that 
\begin{multline}\label{e:TDeltaT-D1}
H_p^{(x_0)}-\mathcal H^{(x_0)}_p\\ =-\frac 1p \sum_{\ell,m=1}^{2n}(g_{x_0}^{\ell m}-\delta^{\ell m})\nabla^{L^p_0}_{e_\ell}\nabla^{L^p_0}_{e_m}+\frac 1p \sum_{\ell=1}^{2n} F_{\ell,{x_0}} \nabla^{L^p_0}_{e_\ell}+V_{x_0}-V_{x_0}(0)+\frac 1pG_{x_0}. 
\end{multline}
By \eqref{e:x0}, we have $g_{x_0}^{\ell m}(Z)=\delta^{\ell m}, \ell,m=1,\ldots,2n$. 

\subsection{Some estimates for the model operator}
In this section, we will prove some norm estimates for the resolvent of the model operator.  
First, we prove an elliptic estimate, taking care of its dependence on $p$. We will denote by $\|\cdot\|$ the $L^2$-norm in $C^\infty_c(T_{x_0}X, E_{x_0})$. Recall that $\{e_j : j=1,\ldots,2n\}$ is a chosen orthonormal base in $T_{x_0}X$. 

\begin{lem}
There exists $C>0$ such that for any $v\in C^\infty_c(T_{x_0}X, E_{x_0})$ and $p\in \NN$, 
\begin{equation}\label{e:ek-el}
\sum_{k,\ell=1}^{2n} \left\|\nabla^{L^p_0}_{e_k}\nabla^{L^p_0}_{e_\ell}v\right\|\leq C\left(\left\|\Delta^{L^p_0}v\right\|+\sqrt{p}\left\|\nabla^{L^p_0}v\right\|\right).
\end{equation}
\end{lem}

\begin{proof}
The operator $\nabla^{L^p_0}_{e_j}$ is formally skew-adjoint in ${L^2(T_{x_0}X, E_{x_0})}$:
\begin{equation}\label{e:adjoint}
\left(\nabla^{L^p_0}_{e_j}\right)^*=-\nabla^{L^p_0}_{e_j},
\end{equation}
and, for the commutator $\left[\nabla^{L^p_0}_{e_j}, \nabla^{L^p_0}_{e_k}\right]$, we have
\begin{equation}\label{e:commutator}
\left[\nabla^{L^p_0}_{e_j}, \nabla^{L^p_0}_{e_k}\right]=pR_{jk},
\end{equation}
where $R_{jk}$ is a constant function. Observe that 
\begin{equation}
\label{e:Delta0} 
\Delta^{L^p_0}=-\sum_{k=1}^{2n} \left(\nabla^{L^p_0}_{e_k}\right)^2.
\end{equation}

By \eqref{e:adjoint}, we have
\[
\|\nabla^{L^p_0}_{e_k}\nabla^{L^p_0}_{e_\ell}v\|^2=\langle\nabla^{L^p_0}_{e_k}\nabla^{L^p_0}_{e_\ell}v, \nabla^{L^p_0}_{e_k}\nabla^{L^p_0}_{e_\ell}v\rangle = \langle  \nabla^{L^p_0}_{e_\ell}(\nabla^{L^p_0}_{e_k})^2\nabla^{L^p_0}_{e_\ell}v,v\rangle.
\]
Now we move $\nabla^{L^p_0}_{e_\ell}$ to the right. Using \eqref{e:adjoint} and \eqref{e:commutator},  we get
\begin{align*}
\|\nabla^{L^p_0}_{e_k}\nabla^{L^p_0}_{e_\ell}v\|^2=& \langle (\nabla^{L^p_0}_{e_k})^2 (\nabla^{L^p_0}_{e_\ell})^2v,v\rangle + 2pR_{\ell k}\langle  \nabla^{L^p_0}_{e_k}\nabla^{L^p_0}_{e_\ell}v,v\rangle \\ = & \langle (\nabla^{L^p_0}_{e_\ell})^2v,(\nabla^{L^p_0}_{e_k})^2v\rangle - 2pR_{\ell k}\langle \nabla^{L^p_0}_{e_\ell}v,\nabla^{L^p_0}_{e_k}v\rangle.
\end{align*}
Since by \eqref{e:Delta0} we have 
\[
\sum_{k,\ell=1}^{2n}\langle (\nabla^{L^p_0}_{e_\ell})^2v,(\nabla^{L^p_0}_{e_k})^2v\rangle=\langle \sum_{\ell=1}^{2n} (\nabla^{L^p_0}_{e_\ell})^2v,\sum_{k=1}^{2n} (\nabla^{L^p_0}_{e_k})^2v\rangle=\left\|\Delta^{L^p_0}v\right\|^2,
\]
this immediately completes the proof.
\end{proof}

Since $\mathcal H^{(x_0)}_{p}$ is self-adjoint and its spectrum coincides with $\Sigma_{x_0}$, its resolvent $R^{(x_0)}_{p}(\lambda):=\left(\mathcal H^{(x_0)}_{p}-\lambda\right)^{-1}$ satisfies
\begin{equation}\label{e:res1-s1}
\left\|R^{(x_0)}_{p}(\lambda)\right\|\leq d(\lambda,\Sigma)^{-1},\quad \lambda\not\in \Sigma,
\end{equation}
where $\|\cdot\|$ denotes the operator norm for the $L^2$-norms. 

Next, assuming $\lambda\not\in \Sigma$, $|\lambda|<K$, we get
\begin{equation}\label{e:res2-s1}
\left\|\tfrac{1}{p}\Delta^{L^p_0}R^{(x_0)}_{p}(\lambda)\right\|=\left\|I+(\lambda-V(x_0))R^{(x_0)}_{p}(\lambda) \right\| \leq K_1d(\lambda,\Sigma)^{-1},
\end{equation}
where $K_1=2K+\sup_{x\in X}|V(x)|$. (Here we use the fact that $d(\lambda,\Sigma)\leq |\lambda|\leq K$ and, therefore, $1\leq Kd(\lambda,\Sigma)^{-1}$.)

Finally, for any $v\in C^\infty_c(T_{x_0}X, E_{x_0})$, we have
\begin{multline*}
\left\|\frac{1}{\sqrt{p}}\nabla^{L^p_0}R^{(x_0)}_{p}(\lambda)v\right\|^2=-\left\langle\frac{1}{p}\Delta^{L^p_0}R^{(x_0)}_{p}(\lambda)v,R^{(x_0)}_{p}(\lambda)v \right\rangle \\
\\ \leq \left\|\frac{1}{p}\Delta^{L^p_0}R^{(x_0)}_{p}(\lambda)v\right\| \left\| R^{(x_0)}_{p}(\lambda)v \right\|,
\end{multline*}
that, by \eqref{e:res1-s1} and \eqref{e:res2-s1},  gives the estimate
\begin{equation}\label{e:res3-s1}
\left\|\frac{1}{\sqrt{p}}\nabla^{L^p_0}R^{(x_0)}_{p}(\lambda)\right\|\leq K_1d(\lambda,\Sigma)^{-1}. 
\end{equation}

By \eqref{e:ek-el}, \eqref{e:res2-s1} and \eqref{e:res3-s1}, we get
\begin{multline}\label{e:ek-el-res}
\sum_{k,\ell=1}^{2n} \left\|\frac 1p\nabla^{L^p_0}_{e_k}\nabla^{L^p_0}_{e_\ell}R^{(x_0)}_{p}(\lambda) \right\|\\ \leq C\left(\left\|\frac 1p \Delta^{L^p_0}R^{(x_0)}_{p}(\lambda)\right\|+\left\|\frac{1}{\sqrt{p}}\nabla^{L^p_0}R^{(x_0)}_{p}(\lambda)\right\|\right) \leq C_1d(\lambda,\Sigma)^{-1}.
\end{multline}

\subsection{Construction of an approximate inverse}
For each $p\in \NN$, we consider the restrictions of the coordinates charts $\varkappa_{x_0}$ to the ball $B(0,p^{-1/4})$. We can choose a finite collection of coordinate charts 
\[
\varkappa_{\alpha,p}:=\varkappa_{x_{\alpha,p}}\left|_{B(0,p^{-1/4})}\right. : B(0,p^{-1/4}) \to U_{\alpha,p} := \varkappa_{\alpha,p}(B(0,p^{-1/4}))\subset X.  
\] 
with $1\leq \alpha\leq I_p$, $I_p\in \NN\cup \{\infty\}$, such that the open subsets $U_{\alpha,p}$ cover $X$. For simplicity of notation, we will often omit $p$, writing $\varkappa_{\alpha}$, $U_\alpha$ etc. 
By a classical argument similar to Vitali's lemma, one can show for the cardinality of the set $\mathcal I_{p,\alpha}= \{1\leq \beta \leq I_p : U_\alpha \cap U_\beta \neq \varnothing \}$,  
\begin{equation}\label{e:Leb}
\# \mathcal I_{p,\alpha} \leq K_0, \quad 1\leq \alpha\leq I_p,
\end{equation}
with the constant $K_0$, depending only on the dimension $n$.

Choose a family of smooth functions $\{\varphi_{\alpha}=\varphi_{\alpha,p} : \RR^{2n}\to [0; 1], 1\leq \alpha\leq I_p\}$ supported on the ball $B(0,p^{-1/4})$, which gives a partition of unity on $X$ subordinate to $\{U_\alpha\}$:
\[
\sum_{\alpha=1}^{I_p}\varphi_\alpha \circ \varkappa^{-1}_\alpha \equiv 1\ \text{on}\ X,
\]
and satisfies the condition: for any $\gamma\in \ZZ^{2n}_+$, there exists $C_\gamma>0$ such that
\[
|\partial^\gamma\varphi_\alpha(Z)|<C_\gamma p^{(1/4)|\gamma|}, \quad Z\in \RR^{2n}, \quad 1\leq \alpha\leq I_p.
\] 

For every $1\leq \alpha\leq I_p$, we denote by $g_\alpha$ the induced Riemannian metric $g_{x_\alpha}$ on $B(0,p^{-1/4})$. We will use notation
\[
T^*_{\alpha}=T^*_{\alpha,p}: C^\infty(X, L^p\otimes E)\to C^\infty(B(0,p^{-1/4}), E_{x_\alpha})
\]
for the composition of the map $T^*_{x_\alpha,p}$ defined by \eqref{e:defT} with the restriction map  $C^\infty(B(0,c), E_{x_\alpha})\to  C^\infty(B(0,p^{-1/4}), E_{x_\alpha})$.

We have
\begin{equation}\label{e:Tap-unitary}
\|T^*_{\alpha}u\|^2_{L^2(B(0,p^{-1/4}), E_{x_\alpha})}=\|u\|^2_{L^2(U_\alpha,L^p\otimes E)}.
\end{equation}

If $U_\alpha \cap U_\beta \neq \varnothing$, we denote by $\varkappa_{\beta,\alpha} :=\varkappa^{-1}_\beta\circ \varkappa_\alpha : \varkappa^{-1}_\alpha(U_\alpha \cap U_\beta) \to \varkappa^{-1}_\beta(U_\alpha \cap U_\beta)$ and $\tau_{\alpha,\beta,p} :=\tau^{-1}_{\alpha,p}\circ \tau_{\beta,p} : (U_\alpha \cap U_\beta)\times E_{x_\beta}\to (U_\alpha \cap U_\beta)\times E_{x_\alpha}$ the associated coordinate change transformations. The family $\varkappa_{\beta,\alpha}$, $U_\alpha \cap U_\beta \neq \varnothing$,
 for any $\gamma\in \ZZ^{2n}_+$ satisfies
\[
|\partial^\gamma \varkappa_{\beta,\alpha}(Z)| < C_\gamma, \quad Z\in \varkappa^{-1}_\alpha(U_\alpha \cap U_\beta), \quad 1\leq \alpha,\beta\leq I_p.
\]  
We can write each $\tau_{\alpha,\beta}$
as 
\[
\tau_{\alpha,\beta}(x,v)=(x, \tau_{\alpha,\beta}(x)v), \quad (x,v)\in (U_\alpha \cap U_\beta)\times E_{x_\beta},
\]
where $\tau_{\alpha,\beta}(x): E_{x_\beta}\to E_{x_\alpha}$ is a linear unitary operator.
Then for any $u\in C^\infty(X,L^p\otimes E)$ and for any $1\leq \alpha,\beta\leq I_p$ with $U_\alpha \cap U_\beta \neq \varnothing$, we have 
\begin{equation}\label{e:Tap-Tbp}
T^*_{\alpha}u(Z) =\tau_{\alpha,\beta}(\varkappa_\alpha(Z)) J_{\alpha,\beta}(Z)T^*_{\beta}u(\varkappa_{\beta,\alpha}(Z)), \quad Z\in \varkappa^{-1}_\beta(U_\alpha \cap U_\beta),
\end{equation}
where 
\[
J_{\alpha,\beta}(Z)= \frac{|g_\alpha(Z)|^{1/4}}{| g_\beta(\varkappa_{\beta,\alpha}(Z))|^{1/4}}.
\]

Let $\psi : \RR^{2n}\to [0,1]$ be a smooth function such that $\psi(Z)=1$ for $|Z|\leq 1$,   
$\psi(Z)=0$ for $|Z|\geq 2$. Put $\psi_p(Z)=\psi(p^{1/4}Z)$. Observe that $\psi_p\varphi_\alpha=\varphi_\alpha$ for any $\alpha$. 

For any $\lambda\not\in \Sigma$ and $p\in \NN$, the operator $Q_p(\lambda)$ acting on $C^\infty(X, L^p\otimes E)$ is defined for any $u\in C^\infty(X, L^p\otimes E)$ by 
\begin{equation}\label{e:defQp0}
Q_p(\lambda)u=\sum_{\beta=1}^{I_p}(T^*_{\beta})^{-1}\left(\psi_p\circ R^{(x_\beta)}_{p}(\lambda) \circ \varphi_\beta\right)T^*_{\beta} u.  
\end{equation}
It is easy to see that $Q_p(\lambda)$ is a bounded operator in $L^2(X, L^p\otimes E)$. 
%For $\alpha =1, \ldots, I_p$, we have
%\begin{multline}\label{e:defQp}
%T^*_{\alpha} Q_p(\lambda)u(Z)=\sum_{\beta\in \mathcal I_{p,\alpha}}\tau_{\alpha,\beta}(\kappa_\alpha(Z))J_{\alpha,\beta}(Z)\times \\\times \left[\left(\chi_p\circ R^{(x_\beta)}_{p}(\lambda) \circ \psi_\beta\right)T^*_{\beta} u\right](\kappa_{\beta,\alpha}(Z)), \quad Z\in B(0,p^{-1/2+\theta}). 
%\end{multline}

\subsection{Proof of Proposition \ref{p:Klambda}}
Let $u\in C^\infty(X, L^p\otimes E)$. Using \eqref{e:Tap-Tbp} and \eqref{e:defQp0}, for any $\alpha =1, \ldots, I_p$, we have
\begin{multline*}
T^*_{\alpha} \left(H_{p}-\lambda\right)Q_p(\lambda)u(Z)\\
\begin{aligned}
= & \sum_{\beta=1}^{I_p}T^*_{\alpha} \left(H_{p}-\lambda\right) (T^*_{\beta})^{-1}\left(\psi_p\circ R^{(x_\beta)}_{p}(\lambda) \circ \varphi_\beta\right)T^*_{\beta} u(Z) \\
= & \sum_{\beta\in \mathcal I_{p,\alpha}}\tau_{\alpha,\beta}(\varkappa_\alpha(Z))J_{\alpha,\beta}(Z)\left(H_p^{(x_\beta)}-\lambda\right) \times \\ & \times \left(\psi_p\circ R^{(x_\beta)}_{p}(\lambda) \circ \varphi_\beta \right)T^*_{\beta}u(\varkappa_{\beta,\alpha}(Z)).
\end{aligned}
\end{multline*}
We can write 
\[
T^*_{\alpha} \left(H_{p}-\lambda\right)Q_p(\lambda)u=T^*_{\alpha}(u+K_p(\lambda)u),
\]
where 
\begin{equation}\label{e:Kp}
T^*_{\alpha} (K_p(\lambda)u) = R_{1,\alpha}u+R_{2,\alpha}u,
\end{equation}
with
\begin{multline*}
R_{1,\alpha}u(Z) = \sum_{\beta\in \mathcal I_{p,\alpha}} \tau_{\alpha,\beta}(\varkappa_\alpha(Z))J_{\alpha,\beta}(Z)\times \\ \times \left(\psi_p\circ (H_p^{(x_\beta)}-\mathcal H^{(x_\beta)}_{p}) \circ R^{(x_\beta)}_{p}(\lambda) \circ \varphi_\beta\right)T^*_{\beta}u(\varkappa_{\beta,\alpha}(Z)),
\end{multline*}
\begin{multline*}
R_{2,\alpha}u(Z) = \sum_{\beta\in \mathcal I_{p,\alpha}} \tau_{\alpha,\beta}(\varkappa_\alpha(Z))J_{\alpha,\beta}(Z) \times\\ \times\left([H_p^{(x_\beta)},\psi_p]\circ R^{(x_\beta)}_{p}(\lambda) \circ \varphi_\beta\right)T^*_{\beta}u(\varkappa_{\beta,\alpha}(Z)).
\end{multline*}

Since $\psi_p$ is supported on the ball $B(0,2p^{-1/4})$, we have 
\[
|g^{\ell m}_\beta(Z)-\delta^{\ell m}|\leq Cp^{-1/4}, \quad |V_\beta(Z)-V_\beta(0)|\leq Cp^{-1/4}, 
%|\Gamma^E(e_\ell)(Z)|\leq Cp^{-1/2+\theta}.
\] 
on the support of $\psi_p$ and therefore using  \eqref{e:ek-el}, \eqref{e:res1-s1}, \eqref{e:res2-s1}, \eqref{e:res3-s1}, \eqref{e:Tap-unitary} and \eqref{e:TDeltaT-D1}, we conclude: for any $u\in C^\infty(X, L^p\otimes E)$, we get
\begin{multline*}
\left\|\psi_p (H_p^{(x_\beta)}-\mathcal H^{(x_\beta)}_{p})R^{(x_\beta)}_{p}(\lambda)\varphi_\beta T^*_{\beta}u\right\|_{L^2(\RR^{2n}, E_{x_0})}\\ 
\begin{aligned}
\leq & Cp^{-1/4}\sum_{\ell,m=1}^{2n}\left\| \frac 1p\nabla^{L^p_0}_{e_\ell}\nabla^{L^p_0}_{e_m} R^{(x_\beta)}_{p}(\lambda) \circ \varphi_\beta T^*_{\beta}u\right\|_{L^2(\RR^{2n}, E_{x_0})}\\ & +Cp^{-1/2}\left\|\frac{1}{\sqrt{p}}\nabla^{L^p_0}R^{(x_\beta)}_{p}(\lambda) \circ \varphi_\beta T^*_{\beta}u\right\|_{L^2(\RR^{2n}, E_{x_0})}\\  &+Cp^{-1/4}\left\|R^{(x_\beta)}_{p}(\lambda) \circ \varphi_\beta T^*_{\beta}u\right\|_{L^2(\RR^{2n}, E_{x_0})},
\end{aligned}
\end{multline*}
and 
\begin{equation}
\|R_{1,\alpha}u\|_{L^2(\RR^{2n})} \leq  Cp^{-1/4}d(\lambda,\Sigma)^{-1} \sum_{\beta\in \mathcal I_{p,\alpha}} \left\|u\right\|_{L^2(U_\beta,L^p)}. \label{e:R2}
\end{equation}

By \eqref{e:TDeltaT-D}, for the commutator $[H_p^{(x_\beta)}, \psi_p]$, we get
\[
[H_p^{(x_\beta)}, \psi_p] =-\frac 1p\sum_{\ell,m=1}^{2n}(2g^{\ell m}_\beta \nabla_{e_\ell}\psi_p\nabla^{L^p_0}_{e_m}+g^{\ell m}_\beta\nabla_{e_\ell,e_m}^2\psi_p) +\frac 1p \sum_{\ell=1}^{2n} F_{\ell,\beta} \nabla_{e_\ell}\psi_p.
\]
Since $|\nabla \psi_p|<C p^{1/4}$, $|\nabla^2\psi_p|<Cp^{1/2}$,
using \eqref{e:res1-s1}, \eqref{e:res2-s1}, \eqref{e:res3-s1} and \eqref{e:Tap-unitary} as above, we conclude:
\begin{align}
\|R_{2,\alpha}u\|_{L^2(\RR^{2n})} \leq & Cp^{-1/2}\sum_{\beta\in \mathcal I_{p,\alpha}} \left\|R^{(x_\beta)}_{p}(\lambda)\varphi_\beta T^*_{\beta}u \right\|_{L^2(\RR^{2n}, E_{x_0})}\nonumber \\ & +C p^{-1/4}\sum_{\beta\in \mathcal I_{p,\alpha}} \left\|\frac{1}{\sqrt{p}}\nabla^{L^p_0}R^{(x_\beta)}_{p}(\lambda) \varphi_\beta T^*_{\beta}u\right\|_{L^2(\RR^{2n}, E_{x_0})}\nonumber \\
\leq & Cp^{-1/4} d(\lambda,\Sigma)^{-1} \sum_{\beta\in \mathcal I_{p,\alpha}} \left\|u\right\|_{L^2(U_\beta,L^p)}. \label{e:R1}
\end{align}
By \eqref{e:Kp}, \eqref{e:R2} and \eqref{e:R1}, we arrive at the following estimate:
\[
\|T^*_{\alpha}(K_p(\lambda)u)\|_{L^2(\RR^{2n}, E_{x_0})} \leq Cp^{-1/4}d(\lambda,\Sigma)^{-1}\sum_{\beta\in \mathcal I_{p,\alpha}} \left\|u\right\|_{L^2(U_\beta,L^p\otimes E)}.
\]
By \eqref{e:Tap-unitary}, it follows that
\[
\|K_p(\lambda)u\|_{L^2(U_\alpha,L^p\otimes E)} \leq Cp^{-1/4}d(\lambda,\Sigma)^{-1}\sum_{\beta\in \mathcal I_{p,\alpha}} \left\|u\right\|_{L^2(U_\beta,L^p\otimes E)},
\]
and, since $(\sum_{\beta\in \mathcal I_{p,\alpha}}a_\beta)^2 \leq K_0 \sum_{\beta\in \mathcal I_{p,\alpha}}a^2_\beta$, 
\begin{equation}\label{e:Kp-U-alpha}
\|K_p(\lambda)u\|^2_{L^2(U_\alpha,L^p)\otimes E} \leq C^2K_0p^{-1/2}d(\lambda,\Sigma)^{-2}\sum_{\beta\in \mathcal I_{p,\alpha}} \left\|u\right\|^2_{L^2(U_\beta,L^p\otimes E)}.
\end{equation}
Using the fact that $\{U_\alpha : \alpha =1, \ldots, I_p\}$ is a covering of $X$ and \eqref{e:Kp-U-alpha}, we infer 
\begin{multline*}
\|K_p(\lambda) u\|^2_{L^2(X,L^p\otimes E)}\leq  \sum_{\alpha=1}^{I_p} \|K_p(\lambda)u\|^2_{L^2(U_\alpha,L^p\otimes E)}\\ \leq C^2K_0p^{-1/2}d(\lambda,\Sigma)^{-2} \sum_{\alpha=1}^{I_p} \sum_{\beta\in \mathcal I_{p,\alpha}} \left\|u\right\|^2_{L^2(U_\beta,L^p\otimes E)},
\end{multline*}
Using \eqref{e:Leb}, it is easy to see that each term in the double sum in the right rand-side of the last estimate enters at most $K_0$ times. Therefore, we infer that 
\[
\sum_{\alpha=1}^{I_p} \sum_{\beta\in \mathcal I_{p,\alpha}} \left\|u\right\|^2_{L^2(U_\beta,L^p\otimes E)}\leq K_0 \sum_{\alpha=1}^{I_p} \left\|u\right\|^2_{L^2(U_\alpha,L^p\otimes E)}
\]
and 
\[
\|K_p(\lambda) u\|^2_{L^2(X,L^p\otimes E)} \leq C^2K^2_0p^{-1/2}d(\lambda,\Sigma)^{-2} \sum_{\alpha=1}^{I_p} \left\|u\right\|^2_{L^2(U_\alpha,L^p\otimes E)}
\]
Finally, by \eqref{e:Leb}, we have
\[
\sum_{\alpha=1}^{I_p} \left\|u\right\|^2_{L^2(U_\alpha,L^p\otimes E)}\leq  K_0\|u\|^2_{L^2(X,L^p\otimes E)},
\]
that gives  
\[
\|K_p(\lambda) u\|^2_{L^2(X,L^p\otimes E)} \leq C^2K^3_0p^{-1/2}d(\lambda,\Sigma)^{-2} \|u\|^2_{L^2(X,L^p\otimes E)}.
\]
This completes the proof of Proposition \ref{p:Klambda}.

\section{Weighted resolvent estimates}\label{s:res-estimates}

In  this section we prove norm weighted estimates for the resolvent of the operator $H_p$. These estimates is a slightly modified version of the estimates obtained in \cite[Theorems 3.3-3.5]{Kor18}, \cite[Theorems 3.2-3.4]{ko-ma-ma}, which are weighted analogs of \cite[Theorems 4.8-4.10]{dai-liu-ma}, \cite[Theorems 1.7-1.9]{ma-ma08}. The main difference is that we state explicitly the $\|\cdot\|^{m,m+2}$-norm estimate for the resolvent instead of the $\|\cdot\|^{-1,1}$-norm estimates for its iterated commutators (see Theorem \eqref{Thm1.9a} below).

\subsection{Preliminaries on Sobolev spaces}
We will need a specific choice of the Sobolev norms adapted to a particular sequence of vector bundles $L^p\otimes E, p\in \mathbb N$ as well as a slightly refined form of the Sobolev embedding theorem. In this section, we collect necessary  information, referring the reader to \cite{Kor91,ko-ma-ma,ma-ma15} for more details. We will keep the setting described in Introduction.

Recall that $dv_{X}$ denotes the Riemannian volume form of $(X,g)$. 
The $L^2$-norm is given by
\begin{equation}\label{e2.5}
\|u\|^2_{p,0}=\int_{X}|u(x)|^2dv_{X}(x), \quad u\in L^2(X,L^p\otimes E).
\end{equation}
For any integer $m>0$, we introduce the norm $\|\cdot\|_{p,m}$ 
on the Sobolev space $H^m(X,L^p\otimes E)$ of order $m$ by the formula
\begin{equation}\label{e2.6}
\|u\|_{p,m}=\left(\sum_{\ell=0}^m \int_{X} \left|\left(\tfrac{1}{\sqrt{p}}
\nabla^{L^p\otimes E}\right)^\ell u(x)\right|^2 dv_{X}(x)\right)^{1/2}.
\end{equation}

Denote by $\langle\cdot,\cdot\rangle_{p,m}$ the corresponding inner 
product on $H^m(X,L^p\otimes E)$. For any integer $m<0$, 
we define the norm in the Sobolev space $H^m(X,L^p\otimes E)$ by duality.
For any bounded linear operator 
$A : H^m(X,L^p\otimes E)\to H^{m^\prime}(X,L^p\otimes E)$, 
$m,m^\prime\in \mathbb Z$, we will denote its operator norm by 
$\|A\|^{m,m^\prime}_p$.

Denote by $C^\infty_b(X, L^p\otimes E)$ the space of smooth sections of $L^p\otimes E$ whose covariant derivatives of any order are uniformly bounded in $X$. So $u\in C^\infty_b(X, L^p_0\otimes E)$ if, for any $k\in \mathbb Z_+$, we have  
\[
\|u\|_{{C}^k_b}:=\sup_{x\in X}\left|\left(\nabla^{L^p\otimes E}\right)^{k}u(x)\right| <\infty.
\]

\begin{prop}[\cite{ma-ma15}, Lemma 2]\label{p:Sobolev}
For any $k, m\in \mathbb N$ with $m>k+n$, we have an embedding
\begin{equation}\label{e2.16}
H^m(X,L^p\otimes E)\subset {C}^k_b(X,L^p\otimes E).
\end{equation}
Moreover, there exists $C_{m,k}>0$ such that, for any $p\in \mathbb N$
and $u\in H^m(X,L^p\otimes E)$,
\begin{equation}\label{e2.17}
\|u\|_{{C}^k_b}\leq C_{m,k}p^{(n+k)/2}\|u\|_{p,m}.
\end{equation}
\end{prop}

For any $y\in X$ and $v\in (L^p\otimes E)_y$, we define the delta-section 
$\delta_v\in \mathscr{C}^{-\infty}(X,L^p\otimes E)$ as a linear functional
on ${C}^{\infty}_c(X,L^p\otimes E)$ given by
\begin{equation}\label{e2.18}
\langle \delta_v, \varphi\rangle 
=\langle v, \varphi(x)\rangle_{h^{L^p\otimes E}}, 
\quad \varphi \in {C}^{\infty}_c(X,L^p\otimes E).
\end{equation}

\begin{prop}[\cite{ko-ma-ma}, Proposition 2.3]\label{p:delta}
For any $m>n$ and $v\in L^p\otimes E$, $\delta_v\in H^{-m}(X,L^p\otimes E)$ with the following norm estimate
\begin{equation}\label{e2.19}
\sup_{|v|=1}p^{-n/2}\|\delta_v\|_{p,-m}<\infty.
\end{equation}
\end{prop}

\subsection{$L^2$-estimates}\label{s:L2estimates}
For the rest of this section, we fix some $\delta>0$ and $K>0$. Denote
\[
\Omega=\Omega_{\delta,K}=\{\lambda\in \CC : d(\lambda,\Sigma)>\delta, |\lambda|<K\}.
\]
By Theorem \ref{t:spectrum}, there exists $p_0\in \NN$ such that, for any $\lambda\in \Omega$ and $p>p_0$, the operator $\lambda-H_p$ is invertible  in $L^2(X,L^p\otimes E)$, and the resolvent $R(\lambda,H_p):=(\lambda-H_p)^{-1}$ satisfies the estimate
\begin{equation}\label{e:res1}
\left\|R(\lambda,H_p)\right\|^{0,0}_p\leq  \frac{1}{\delta}.
\end{equation}
By general elliptic theory, we know that the operator $R(\lambda,H_p)$ defines a bounded operator from  $H^m(X,L^p\otimes E)$ to $H^{m+2}(X,L^p\otimes E)$ for any $m\in \ZZ$.

\begin{prop}
There exists $C>0$ such that for all $\lambda\in \Omega$ and $p>p_0$ we have
\begin{equation}\label{e:res3}
\left\|R(\lambda,H_p)\right\|^{0,2}_p\leq  C. 
\end{equation}
\end{prop}

\begin{proof}
First, we observe that by the definition of the Bochner Laplacian \eqref{e:def-Bochner},
\begin{equation}\label{e:subelliptic-1}
\left\| \nabla^{L^p\otimes E}u\right\|^2=\langle \Delta^{L^p\otimes E}u, u\rangle.
\end{equation}

Using \eqref{e:res1},  we obtain the estimate
\begin{equation}\label{e:Delta-estimate}
\left\|\frac{1}{p}\Delta^{L^p\otimes E}R(\lambda,H_p)\right\| 
= \left\|(\lambda-V) R(\lambda,H_p)+1\right\|\leq C.
\end{equation}
By \eqref{e:res1}, \eqref{e:subelliptic-1} and \eqref{e:Delta-estimate}, we obtain an estimate for the 
$H^1$-norm of $R(\lambda,H_p) u$:
\begin{multline}\label{e:1-estimate}
\left\|R(\lambda,H_p) u\right\|^2_{p,1}=\left\|\frac{1}{\sqrt{p}} \nabla^{L^p\otimes E}R(\lambda,H_p) u\right\|^2+\|R(\lambda,H_p) u\|^2\\ = \left\langle \frac{1}{p} \Delta^{L^p\otimes E}R(\lambda,H_p) u,R(\lambda,H_p) u\right\rangle +\|R(\lambda,H_p)u\|^2\leq C\|u\|^2. 
\end{multline}
Next, we estimate the $H^2$-norm of $R(\lambda,H_p) u$. We will use an equivalent definition of the Sobolev norms, given in terms of local coordinates. For any $x_0\in X$, we will consider normal coordinates $\gamma_{x_0}$  and trivializations of the bundles $L$ and $E$ defined on $B^{X}(x_0,\varepsilon)$ as in Introduction. We still denote by $e_{j}$ the constant vector field $e_j(Z)=e_j, j=1,\ldots,2n$ on $B^{T_{x_0}X}(0,\varepsilon)$. One can show that the restriction of the norm $\|\cdot\|_{p,m}$ to ${C}^\infty_c(B^{T_{x_0}X}(0,\varepsilon), L^p\otimes E)
\cong {C}^\infty_c(B^{X}(x_0,\varepsilon), L^p\otimes E)$ 
is equivalent uniformly on $x_0\in X$ and $p\in \mathbb N$
to the norm $\|\cdot \|^\prime_{p,m}$ given for 
$u\in {C}^\infty_c(B^{T_{x_0}X}(0,\varepsilon), L^p\otimes E)$
by
\begin{equation}\label{localSobolev}
\|u\|^\prime_{p,m}=\left(\sum_{\ell=0}^m\sum_{j_1,\ldots,j_\ell=1}^{2n}
\int_{X_0} \Big(\tfrac{1}{\sqrt{p}}\Big)^\ell|
\nabla^{L_0^p\otimes E_0}_{e_{j_1}}\cdots 
\nabla^{L_0^p\otimes E_0}_{e_{j_\ell}}u|^2dZ\right)^{1/2}.
\end{equation}
That is, there exists $C_m>0$ such that, for any $x_0\in X$, 
$p\in \mathbb N$ 
%and $u\in \mathscr{C}^\infty_c(B^{T_{x_0}X}(0,\varepsilon), 
%L^p\otimes E)
%\cong \mathscr{C}^\infty_c(B^{X}(x_0,\varepsilon), L^p\otimes E)$, 
we have
\begin{equation}\label{pm-prime}
C_m^{-1}\|u\|^\prime_{p,m}\leq \|u\|_{p,m}\leq C_m\|u\|^\prime_{p,m}\,,
\end{equation}
for any $u\in {C}^\infty_c(B^{T_{x_0}X}(0,\varepsilon),
L^p\otimes E)
\cong {C}^\infty_c(B^{X}(x_0,\varepsilon), L^p\otimes E)$.
By choosing an appropriate covering of $X$ by normal coordinate charts, 
we can reduce our considerations to the local setting. Without loss
of generality, we can assume that 
$u\in {C}^\infty_c(B^{T_{x_0}X}(0,\varepsilon), L^p\otimes E)$
for some $x_0\in X$ and the Sobolev norm of $u$ is given by
the norm $\|u\|^\prime_{p,m}$ given by \eqref{localSobolev}.
(Later on, we omit `prime' for simplicity.)

One can write
\begin{equation}\label{Delta-p}
\Delta^{L^p\otimes E}=-\sum_{j,k=1}^{2n}g^{jk}(Z)\left[\nabla^{L^p\otimes E}_{e_j}
\nabla^{L^p\otimes E}_{e_k}- \sum_{\ell =1}^{2n} \Gamma^{\ell}_{jk}(Z)
\nabla^{L^p\otimes E}_{e_\ell}\right],
\end{equation}
where $(g^{jk}(Z))$ is the inverse of the metric tensor and the functions $\Gamma^{\ell}_{jk}, 
$ are defined by $\nabla^{TX}_{e_j}e_k=\sum_{\ell}\Gamma^{\ell}_{jk}e_\ell$. We also
observe that 
\begin{equation}\label{e:adj}
(\nabla^{L^p\otimes E}_{e_k})^*=-\nabla^{L^p\otimes E}_{e_k}+f_k
\end{equation}
for any $k=1,\ldots,2n$ with some $f_k\in C^\infty(X)$. 

By \eqref{e:1-estimate} and \eqref{e:subelliptic-1}, we get
\begin{multline*}
\|R(\lambda,H_p) u\|^2_{p,2}\leq C_1\sum_{j,k=1}^{2n}\left\|\frac{1}{p} \nabla^{L^p\otimes E}_{e_j} \nabla^{L^p\otimes E}_{e_k}R(\lambda,H_p) u\right\|^2+\|R(\lambda,H_p) u\|^2_{p,1}\\
\leq C_1\sum_{k=1}^{2n} \left\langle \frac{1}{p} \Delta^{L^p\otimes E}\frac{1}{\sqrt{p}}\nabla^{L^p\otimes E}_{e_k} R(\lambda,H_p) u, \frac{1}{\sqrt{p}} \nabla^{L^p\otimes E}_{e_k} R(\lambda,H_p) u\right\rangle +C_2\|u\|^2. 
\end{multline*}
The first term in the right-hand side of the last inequality can be written as follows:
\begin{multline*}
\sum_{k=1}^{2n} \left\langle \frac{1}{p} \Delta^{L^p\otimes E}\frac{1}{\sqrt{p}}\nabla^{L^p\otimes E}_{e_k} R(\lambda,H_p) u, \frac{1}{\sqrt{p}} \nabla^{L^p\otimes E}_{e_k} R(\lambda,H_p) u\right\rangle \\  =\sum_{k=1}^{2n} \left\langle \frac{1}{\sqrt{p}}\nabla^{L^p\otimes E}_{e_k} \frac{1}{p} \Delta^{L^p\otimes E} R(\lambda,H_p) u, \frac{1}{\sqrt{p}} \nabla^{L^p\otimes E}_{e_k} R(\lambda,H_p) u\right\rangle \\
+\sum_{k=1}^{2n} \left\langle \left[\frac{1}{p} \Delta^{L^p\otimes E}, \frac{1}{\sqrt{p}}\nabla^{L^p\otimes E}_{e_k}\right] R(\lambda,H_p) u, \frac{1}{\sqrt{p}} \nabla^{L^p\otimes E}_{e_k} R(\lambda,H_p) u\right\rangle \\ = I_{1}+I_{2}.
\end{multline*}
For the $I_{1}$-term, using \eqref{e:Delta-estimate}, \eqref{e:1-estimate} and \eqref{e:adj}, we get
\begin{align*}
I_{1}= & \sum_{k=1}^{2n} \left\langle \frac{1}{p} \Delta^{L^p\otimes E} R(\lambda,H_p) u, \left(\frac{1}{\sqrt{p}}\nabla^{L^p\otimes E}_{e_k}\right)^*\frac{1}{\sqrt{p}} \nabla^{L^p\otimes E}_{e_k} R(\lambda,H_p) u\right\rangle \\
\leq & C_3\left\|R(\lambda,H_p) u\right\|_{p,2} \|u\|.
\end{align*}

The commutator $\left[\frac{1}{p}\Delta^{L^p\otimes E}, \frac{1}{\sqrt{p}}\nabla^{L^p\otimes E}_{e_{k}}\right]$ is a second order differential operator whose coefficients are uniformly bounded in $x_0\in X$ (cf. \eqref{e:comm} below). By \eqref{e:1-estimate}, it follows that 
\[
I_{2}\leq C_4\left\|R(\lambda,H_p) u\right\|_{p,2} \|u\|.
\]

Combining the above estimates, we conclude that
\[
\|R(\lambda,H_p) u\|^2_{p,2}\leq C_5\left\|R(\lambda,H_p) u\right\|_{p,2} \|u\|+C_6\|u\|^2.
\]
Applying the inequality $ab\leq \frac 12(\epsilon^2a^2+\epsilon^{-2}b^2)$ with a suitable $\epsilon>0$ to the first term in the right-hand side of this inequality, we complete the proof.
\end{proof}

\subsection{Weighted estimates}
We will consider a sequence of  weight functions $\Phi_p\in C^\infty(X)$, $p\in \NN$, satisfying the following condition: for any integer $k>0$ there exists $C_k>0$ such that
\begin{equation}\label{e:chi-p}
\left(\frac{1}{\sqrt{p}}\right)^{k-1}\left|\nabla^k\Phi_p(x)\right|<C_k, \quad x\in X. 
\end{equation} 

Define a family of differential operators on $C^\infty(X,L^p\otimes E)$ by
\begin{align}\label{e:weight-operator}
H_{p;\alpha}:= e^{\alpha \Phi_p} H_p e^{-\alpha \Phi_o}, \quad \alpha\in \RR.
\end{align}
An easy computation gives that
\begin{equation}\label{e:DpaW}
H_{p;\alpha}=H_p+\frac 1p(\alpha A_{p}+\alpha^2B_{p}),
\end{equation}
where 
\begin{equation} \label{e:ApBp}
A_{p}=-2\nabla \Phi_p\cdot \nabla^{L^p\otimes E} +\Delta \Phi_p,, \quad B_{p}=-|\nabla \Phi_p|^2.
\end{equation}
From \eqref{e:ApBp}, we immediately infer that, for any $m\in \mathbb N$, 
there exists $C_m>0$ such that, for any $p\in \mathbb N$
and $u\in H^{m}(X,L^p\otimes E)$,
\begin{equation}\label{e:Sobolev-mapping}
\|A_{p}u\|_{p,m-1}\leq C_mp^{1/2}\|u\|_{p,m},\quad \|B_{p}u\|_{p,m}
\leq C_m\|u\|_{p,m}.
\end{equation}
Moreover, $C_m$ depends on $C^{m+2}$-norm of $\Phi_p$ in \eqref{e:chi-p}.

The following theorem is a refinement of \cite[Theorem 3.4]{ko-ma-ma}. Recall that $p_0\in \NN$ is defined in the beginning of Section \ref{s:L2estimates}. 

\begin{thm}\label{Thm1.9a}
There exists $c_0>0$ such that, for any 
$p\in \mathbb N$, $p>p_0$, $\lambda\in \Omega$,  
and $|\alpha|<c_0\sqrt{p}$,
the operator $\lambda-H_{p;\alpha}$
is invertible in $L^2(X, L^p\otimes E)$. Moreover,  for any $m\in \mathbb N$, 
the resolvent $\left(\lambda-H_{p;\alpha}\right)^{-1}$ 
maps $H^m(X, L^p\otimes E)$ to $H^{m+2}(X, L^p\otimes E)$ with the following norm estimates:
\begin{equation}\label{e:mm+2}
\left\|\Big(\lambda-H_{p;\alpha}\Big)^{-1}
\right\|_{p}^{m,m+2}\leq C_m,
\end{equation}
where $C_m>0$ is independent of $p\in \mathbb N$, $p>p_0$, $\lambda\in \Omega$,  and $|\alpha|<c_0\sqrt{p}$.
\end{thm}

\begin{proof}%[Proof of Theorem~\ref{Thm1.9a}]
As above, we denote $R(\lambda, H_p)=(\lambda-H_p)^{-1}$.
We can write
\[
\Big(\lambda-H_{p;\alpha}\Big)R(\lambda,H_p)=1-(H_{p;\alpha}-H_p)R(\lambda,H_p).
\]

By \eqref{e:DpaW}, \eqref{e:Sobolev-mapping} and \eqref{e:res3}, 
it follows that, for all $\lambda\in \Omega$, $p\in \mathbb N$ and
$\alpha\in \mathbb R$, we have
\begin{multline}
\left\|(H_{p;\alpha}-H_p)R(\lambda,H_p)\right\|^{0,0}_p\\
\leq  C\left(\frac{|\alpha|}{\sqrt{p}}
\left\|R(\lambda,H_p)\right\|^{0,1}_p
+\frac{\alpha^2}{p}\left\|R(\lambda,H_p
\right\|^{0,0}_p\right) \leq C\left(\frac{|\alpha|}{\sqrt{p}}+\frac{\alpha^2}{p}\right),
\end{multline}
where $C>0$ is some constant.

From now on, we will assume that $c_0>0$ satisfies $C(c_0+c_0^2)<\frac 12$. 
Then, if $|\alpha|<c_0\sqrt{p}$, we have
\begin{equation}\label{e:d}
\left\|(H_{p;\alpha}-H_p)
R(\lambda,H_p)\right\|^{0,0}_p
<\frac 12.
\end{equation}
Therefore, for all $\lambda\in \Omega$, $p\in \mathbb N$, 
$\alpha\in \mathbb R$, and $|\alpha|<c\sqrt{p}$,  
the operator $\lambda-H_{p;\alpha}$ is invertible in 
$L^2(X,L^p\otimes E)$, and, for $R(\lambda,H_{p;\alpha}):=(\lambda-H_{p;\alpha})^{-1} $,  we have
\begin{equation}\label{e:res}
R(\lambda,H_{p;\alpha})= R(\lambda,H_p)+R(\lambda, H_{p;\alpha})(H_{p;\alpha}-H_p)R(\lambda, H_p).
\end{equation}
By general elliptic theory, we know that the operator $R(\lambda,H_{p;\alpha})$ defines a bounded operator from  $H^m(X,L^p\otimes E)$ to $H^{m+2}(X,L^p\otimes E)$ for any $m\in \ZZ$. It remains to prove \eqref{e:mm+2}.

By \eqref{e:d} and \eqref{e:res}, we get
\begin{multline*}
\left\|R(\lambda,H_{p;\alpha})\right\|^{0,2}_p\leq 
\left\|R(\lambda,H_p)\right\|^{0,2}_p
+\left\| R(\lambda,H_{p;\alpha})
\right\|^{0,2}_p\left\|(H_{p;\alpha}-H_p)R(\lambda,H_p)
\right\|^{0,0}_p \\ 
\leq  \left\|R(\lambda,H_p)\right\|^{0,2}_p
+\frac 12\left\| R(\lambda,H_{p;\alpha})
\right\|^{0,2}_p,
\end{multline*}
that gives 
\[
\left\|R(\lambda,H_{p;\alpha})\right\|^{0,2}_p\leq  2\left\|R(\lambda,H_p)\right\|^{0,2}_p
\]
and, by \eqref{e:res3}, proves \eqref{e:mm+2} for $m=0$.

Now we again will work locally. Let $\{e_j\}$ be a local frame of vector fields on $X$.
By \eqref{localSobolev}, we see that for any $k\geq 1$ there exists $C_k>0$ such that 
\begin{equation}\label{e:k-k-1}
\|v\|_{p,k} \leq C_k\left(\sum_{j=1}^{2n}\left\|\left(\frac{1}{\sqrt{p}}\nabla^{L^p\otimes E}_{e_{j}}\right)v\right\|_{p,k-1}+\|v\|_{p,k-1}\right) 
\end{equation}
for any $v\in {C}^\infty_c(X, L^p\otimes E).$

For any $1\leq j \leq 2n$ and $u\in {C}^\infty_c(X, L^p\otimes E).$, we can write 
\begin{multline}\label{e:estj-k}
\Big(\frac{1}{\sqrt{p}}\nabla^{L^p\otimes E}_{e_{j}}\Big) 
R(\lambda,H_{p;\alpha})u=R(\lambda,H_{p;\alpha})\Big(\frac{1}{\sqrt{p}}\nabla^{L^p\otimes E}_{e_{j}}\Big)u\\ +R(\lambda,H_{p;\alpha})\left[\frac{1}{\sqrt{p}}\nabla^{L^p\otimes E}_{e_{j}}, H_{p;\alpha}\right]  R(\lambda,H_{p;\alpha})u.  
\end{multline}
that gives the estimate
\begin{multline}\label{e:estj-k-1}
\left\|\left(\tfrac{1}{\sqrt{p}}\nabla^{L^p\otimes E}_{e_{j}}\right) 
R(\lambda,H_{p;\alpha})u\right\|_{p.m+1}\\ 
\begin{aligned}
\leq & \left\|R(\lambda,H_{p;\alpha})\right\|^{m-1,m+1}_{p}\left\|\left(\tfrac{1}{\sqrt{p}}\nabla^{L^p\otimes E}_{e_{j}}\right)u\right\|_{p,m-1}\\ & +\left\|R(\lambda,H_{p;\alpha})\right\|_p^{m-1.m+1} \left\|\left[\tfrac{1}{\sqrt{p}}\nabla^{L^p\otimes E}_{e_{j}}, H_{p;\alpha}\right]\right\|_p^{m+1.m-1}\times \\ & \times \left\|R(\lambda,H_{p;\alpha})\right\|_p^{m-1.m+1} \|u\|_{p,m-1}. 
\end{aligned} 
\end{multline}

As in \cite[Proposition 3.5]{ko-ma-ma}, for any $1\leq j \leq 2n$, the commutator $\left[\frac{1}{\sqrt{p}}\nabla^{L^p\otimes E}_{e_{j}}, \frac{1}{p}\Delta_{p;\alpha}\right]$ is a second order differential operator of the form
\begin{multline}\label{e:comm}
\left[\frac{1}{\sqrt{p}}\nabla^{L^p\otimes E}_{e_{j}}, 
H_{p;\alpha}\right]=\sum_{i,j}\tilde a^{ij}_{p;\alpha}(Z)
\Big(\frac{1}{\sqrt{p}}\nabla^{L^p\otimes E}_{e_i}\Big)
\Big(\frac{1}{\sqrt{p}}\nabla^{L^p\otimes E}_{e_j}\Big)\\
+\sum_{\ell}\tilde a^{\ell}_{p;\alpha}(Z)\frac{1}{\sqrt{p}}
\nabla^{L^p\otimes E}_{e_\ell}+\tilde a_{p;\alpha}(Z),
\end{multline}
whose coefficients $\tilde a_{p;\alpha}^{ij}$, $\tilde a_{p;\alpha}^\ell$ and $\tilde a_{p;\alpha}$, bounded uniformly on $p\in \mathbb N$ and $|\alpha|<c\sqrt{p}$.

Using \eqref{e:k-k-1}, \eqref{e:estj-k-1} and \eqref{e:comm}, we infer that for any $m\geq 1$
\begin{multline}
\|R(\lambda,H_{p;\alpha})\|^{m, m+2}_{p}\\ \leq C_m\left(\|R(\lambda,H_{p;\alpha})\|^{m-1,m+1}_{p}+(\|R(\lambda,H_{p;\alpha})\|_p^{m-1.m+1})^2\right),
\end{multline}
that proves recursively \eqref{e:mm+2} for all $m\geq 0$. 
\end{proof}

\section{Spectral projection}\label{s:projection}

Let us consider an interval $I=(\alpha,\beta)$ such that $\alpha,\beta\not \in \Sigma$.  Let $P_{p,I}$ be the spectral projection of the operator $H_{p}$ associated with $I$  and $P_{p,I}(x,x^\prime)$, $x,x^\prime\in X$, be its smooth kernel with respect to the Riemannian volume form $dv_X$. 
In this section, we will study the asymptotic behavior of $P_{p,I}(x,x^\prime)$ as $p\to \infty$. 

 By Theorem \ref{t:spectrum}, there exists $\mu_0>0$ and $p_0\in \NN$ such that for any $p>p_0$ 
\[
\sigma(H_{p})\subset (-\infty, \alpha-\mu_0) \cup I \cup (\beta+\mu_0, \infty).
\]
Let $\Gamma$ be the boundary of the rectangle $\Pi=(\alpha-\mu_0/2,\beta+\mu_0/2)+i(-\mu_0/2, \mu_0/2)$ in $\CC$ counterclockwise oriented.  
Then for any integer $m>0$ and $p>p_0$, we can write
\begin{equation}\label{e:Pqp-integral}
P_{p,I}=\frac{1}{2\pi i}
\int_\Gamma \lambda^{m-1}(\lambda-H_{p})^{-m}d\lambda. 
\end{equation}

\subsection{Off-diagonal estimates}\label{s:far-off-diagonal}
The proof of Theorem~\ref{t:exp-Pp} closely follows the proof of \cite[Theorem 1.2]{ko-ma-ma}, so we just give a short outline.

As shown in \cite[Proposition 4.1]{Kor91} (see also \cite[Section 3.1]{ko-ma-ma}), for any $p\in \mathbb N$, there exists a function $\widetilde{d}_p$, satisfying the following conditions:

(1) we have
\begin{equation}\label{(1.1)}
\vert \widetilde{d}_p(x,y) - d (x,y)\vert  < \frac{1}{\sqrt{p}}\;,
\quad x, y\in X;
\end{equation}

(2) for any $k>0$, there exists $c_k>0$ such that
\begin{equation}
\label{dist}
\left(\frac{1}{\sqrt{p}}\right)^{k-1} \left| \nabla^k_{x}
\widetilde{d}_p(x,y)\right| < c_{k}\:,\quad x, y\in X.
\end{equation}
 We get a family $\{\widetilde{d}_{p,y} : y\in X\}$  of weight functions on $X$ given by
\begin{equation}\label{e:3.10}
\widetilde{d}_{p,y}(x) = \widetilde{d}_p(x,y), \quad  x\in X,
\end{equation}
which satisfy \eqref{e:chi-p} uniformly on $y\in X$.

As in \eqref{e:weight-operator}, consider the family of differential operators
\[
H_{p;\alpha,y}:= e^{\alpha \widetilde{d}_{p,y}}
H_{p} e^{-\alpha \widetilde{d}_{p,y}} \quad \alpha\in \RR, \quad y\in X.
\]
By Theorem~\ref{Thm1.9a},  we get Sobolev norm estimates, uniform in $y$, for the operator $(\lambda-H_{p;\alpha,y})^{-m}$  for any $m\in \NN$. Next, we derive pointwise exponential estimates for the  Schwartz kernel of this operator and its derivatives of an arbitrary order, using a refined form of 
the Sobolev embedding theorem stated in Propositions~\ref{p:Sobolev} and \ref{p:delta}. 
Finally, we use the formula \eqref{e:Pqp-integral} to complete the proof of Theorem~\ref{t:exp-Pp}.

\subsection{Localization of the problem}\label{local}
Next, we localize the problem, slightly modifying the constructions of \cite[Sections 1.1 and 1.2]{ma-ma08}. 

We fix $x_0\in X$. We will use normal coordinates and trivializations of the bundles $L$ and $E$ defined on $B^{X}(x_0,\varepsilon)$ as in Introduction. The fixed orthonormal basis $\{e_j\}$ of $T_{x_0}X$ gives rise to an isomorphism $X_0:=\mathbb R^{2n}\cong T_{x_0}X$. Consider the trivial bundles $L_0$ and $E_0$  on $X_0$ with fibers $L_{x_0}$ and $E_{x_0}$, respectively. The above identifications induce the Riemannian metric $g$ on $B^{T_{x_0}X}(0,\varepsilon)$ as well as the connections $\nabla^L$ and $\nabla^E$ and the Hermitian metrics $h^L$ and $h^E$ on the restrictions of $L_0$ and $E_0$ to $B^{T_{x_0}X}(0,\varepsilon)$. In particular, $h^L$, $h^E$ are the constant metrics $h^{L_0}=h^{L_{x_0}}$, $h^{E_0}=h^{E_{x_0}}$. For some $\varepsilon \in (0,r_X/4)$, which will be fixed later, we extend these geometric objects from $B^{T_{x_0}X}(0,\varepsilon)$ to $X_0\cong T_{x_0}X$ in the following way.  

Let $\rho : \mathbb R\to [0,1]$ be a smooth even function such that $\rho(v)=1$ if $|v|<2$ and $\rho(v)=0$ if $|v|>4$. Let $\varphi_\varepsilon : \mathbb R^{2n}\to \mathbb R^{2n}$ be the map defined by $\varphi_\varepsilon(Z)=\rho(|Z|/\varepsilon)Z$. Set $\nabla^{E_0}=\varphi^*_\varepsilon\nabla^E$. Define a Hermitian connection $\nabla^{L_0}$ on $(L_0,h^{L_0})$ by 
\[
\nabla^{L_0}_u=\nabla^L_{d\varphi_\varepsilon(Z)(u)}+\frac 12(1-\rho^2(|Z|/\varepsilon))R^L_{x_0}(\mathcal R(Z),u). \quad Z\in X_0, \quad u\in T_ZX_0
\]
where we use the canonical isomorphism $X_0\cong T_ZX_0$ and $\mathcal R(Z)=\sum_{j=1}^{2n} Z_je_j\in \mathbb R^{2n}\cong T_ZX_0$. Its curvature is given by \cite[(1.22)]{ma-ma08}
\begin{equation}\label{e:RL0}
\begin{split}
R^{L_0}_Z=& (1-\rho^2(|Z|/\varepsilon))R^L_{x_0}+\rho^2(|Z|/\varepsilon)R^L_{\varphi_\varepsilon(Z)}\\
& -(\rho\rho^\prime)(|Z|/\varepsilon)\frac{\sum_{j=1}^{2n}Z_je^j}{\varepsilon |Z|}\wedge [R^L_{x_0}(\mathcal R,\cdot)-R^L_{\varphi_\varepsilon(Z)}(\mathcal R,\cdot)], 
\end{split}
\end{equation}
where $e^j$ is the dual base in $T^*_{x_0}X\cong T^*_ZX_0$.  

Now we proceed in a slightly different way than in \cite{ma-ma08}.
Recall that, for any $Z\in B^{T_{x_0}X}(0,r_X)\cong B^{X}(x_0,r_X)$, we have a skew-adjoint operator $B_Z : T_ZX_0\to T_ZX_0$ such that 
\[
iR^L_Z(u,v)=g_Z(B_Zu,v), \quad u,v\in T_ZX_0. 
\]
Its eigenvalues have the form $\pm i a_j(Z), j=1,\ldots,n,$ with $a_j(Z)>0$. We set 
\[
B^{X_0}_Z=B_{\varphi_\varepsilon(Z)},\quad V^{X_0}(Z)=V(\varphi_\varepsilon(Z)), \quad Z\in X_0,
\]
and define a symmetric bilinear form  $g^{X_0}_Z$ on $T_ZX_0$ by  
\[
g^{X_0}(u,v)=iR^{L_0}_Z((B^{X_0}_Z)^{-1}u,v), \quad u,v\in T_ZX_0,
\]
By \eqref{e:RL0}, tt is easy to sse that, for $\varepsilon$ small enough, $g^{X_0}$ is positive definite and define a Riemannian metric on $X_0$. From now on, we fix such an $\varepsilon>0$.

Let $\Delta^{L_0^p\otimes E_0}$ be the associated Bochner Laplacian acting on $C^\infty(X_0,L_0^p\otimes E_0)$.  Introduce the operator $H^{X_0}_p$ acting on $C^\infty(X_0,L_0^p\otimes E_0)$ by
\[
H^{X_0}_p=\frac 1p \Delta^{L_0^p\otimes E_0}+ V^{X_0}(Z).
\]
It is clear that, for any $u \in {C}^\infty_c(B^{X_0}(0,2\varepsilon))$, we have
\begin{align}\label{e:4.10}
H_pu(Z)=H^{X_0}_pu(Z).
\end{align}
Moreover, the eigenvalues of $B^{X_0}_Z$ are given by $\pm i a_j(\varphi_\varepsilon(Z)), j=1,\ldots,n$.  Therefore, the set $\Sigma$  for $H^{X_0}_p$ is contained in that for the operators $H_p$. Therefore, $\alpha,\beta\not \in \Sigma$. 

By Theorem \ref{t:spectrum}, there exists $p_0\in \NN$ such that for any $p>p_0$ 
\begin{equation}\label{e:DeltaX0-spectrum}
\sigma(H^{X_0}_p)\subset (-\infty, \alpha-\mu_0) \cup I \cup (\beta+\mu_0, \infty).
\end{equation}
with the same $\mu_0>0$ as above.  Let $P^0_{p,I}$ be the spectral projection of $H^{X_0}_p$ corresponding to the interval $I$ and $P^0_{p,I}(Z,Z^\prime)$, $Z,Z^\prime\in X_0$, be its smooth kernel with respect to the Riemannian volume form $dv_{X_0}$. As in \cite[Theorem 4.1]{ko-ma-ma} (extending \cite[Proposition 1.3]{ma-ma08}), one can show that the kernels $P_{p,I,x_0}(Z,Z^\prime)$ and $P^0_{p,I}(Z,Z^\prime)$ are asymptotically close on $B^{T_{x_0}X}(0,\varepsilon)$ in the $C^\infty$-topology, as $p\to \infty$. 

\begin{thm} \label{p:Pq-difference}
There exists $c_0>$ such that, for any $k\in \mathbb N$, there exists $C_{k}>0$ such that for any $p>p_0$, $x_0\in X$  and $Z,Z^\prime\in B^{X_0}(0,\varepsilon)$,  
 \[
|P_{p,I,x_0}(Z,Z^\prime)-P^0_{p,I}(Z,Z^\prime)|_{C^k}\leq C_{k}e^{-c_0\sqrt{p}}.  
\]
 \end{thm}

 \subsection{Rescaling and formal expansions}\label{scale}
Theorem~\ref{p:Pq-difference} reduces our considerations to the operator family $H^{X_0}_{p}$ acting on $C^\infty(X_0, L_0^p\otimes E_0)\cong C^\infty(\RR^{2n},E_{x_0})$ (parametrized by $x_0\in X$). 

We use the rescaling introduced in \cite[Section 1.2]{ma-ma08}. Let $dv_{X_0}$ be the Riemannian volume form of $(X_0, g^{X_0})$. Then $\kappa_{X_0}$ is the smooth positive function on $X_0$ defined by the equation
\[
dv_{X_0}(Z)=\kappa_{X_0}(Z)dZ, \quad Z\in X_0. 
\] 

Denote $t=\frac{1}{\sqrt{p}}$. For $s\in C^\infty(\mathbb R^{2n}, E_{x_0})$, set
\[
S_ts(Z)=s(Z/t), \quad Z\in \mathbb R^{2n}.
\]
Define the rescaling of the operator $H_p^{X_0}$ by
\begin{equation}\label{scaling}
\mathcal H_t=S^{-1}_t\kappa_{X_0}^{\frac 12}H_p^{X_0}\kappa_{X_0}^{-\frac 12}S_t.
\end{equation}
This is a second order differential operator.
We  expand its coefficients in Taylor series in $t$. For any $m\in \NN$, we get
\begin{equation}\label{e:Ht-formal}
\mathcal H_t=\mathcal H^{(0)}+\sum_{j=1}^m \mathcal H^{(j)}t^j+\mathcal O(t^{m+1}), 
\end{equation}
where there exists $m^\prime\in \NN$ so that for every $k\in\NN$ and $t\in [0,1]$ the derivatives up to order $k$ of the coefficients of the operator $O(t^{m+1})$ are bounded by $Ct^{m+1}(1+|Z|)^{m^\prime}$. 

The leading term $\mathcal H^{(0)}$ is given by \eqref{e:DeltaL0p}. By \cite[Theorem 1.4]{ma-ma08}, the next terms $\mathcal H^{(j)}, j\geq 1,$ have the form
\begin{equation}\label{e:Hj}
\mathcal H^{(j)}=\sum_{k,\ell=1}^{2n} a_{k\ell,j}\frac{\partial^2}{\partial Z_k\partial Z_\ell}+\sum_{k=1}^{2n} b_{k,j}\frac{\partial}{\partial Z_k}+c_{j},
\end{equation}
where $a_{k\ell,j}$ is a homogeneous polynomial in $Z$ of degree $j$,  $b_{kj}$ is a polynomial in $Z$ of degree $\leq j+1$ (of the same parity with $j-1$) and $c_{j}$ is a polynomial in $Z$ of degree $\leq j+2$ (of the same parity with $j$). More precisely, for the operator $H_p=\frac{1}{p}\Delta_p$, the operator $\mathcal H^{(j)}$ coincides with the operator $\mathcal O_j$ introduced in that theorem. In the geeral case, we have 
\[
\mathcal H^{(j)}=\mathcal O_j+\sum_{|\alpha|=j}(\partial^\alpha(V+\tau))_{x_0}\frac{Z^\alpha}{\alpha!},\quad j=1,2,\ldots,
\]
In \cite[Theorem 1.4]{ma-ma08}, explicit formulas are given for $\mathcal O_1$ and $\mathcal O_2$. We refer the reader to \cite{ma-ma08,ma-ma:book} for more details. 

\subsection{Asymptotic expansions of the spectral projection} \label{asymp}
By construction, the operator $\mathcal H_{t}$ is a self-adjoint operator in $L^2(\RR^{2n}, E_{x_0})$, and its spectrum coincides with the spectrum of $H_p^{X_0}$.  By \eqref{e:DeltaX0-spectrum}, there exists $t_0>0$ such that for any $t\in (0,t_0]$
\[
\sigma(\mathcal H_{t})\subset (-\infty, \alpha-\mu_0) \cup I \cup (\beta+\mu_0, \infty).
\]

Let $\mathcal P_{t}$ be the spectral projection of $\mathcal H_t$, corresponding to the interval $I$ and $\mathcal P_{t}(Z,Z^\prime)=\mathcal P_{t,x_0}(Z,Z^\prime)$ be its smooth kernel with respect to $dZ$. 
For any integer $k>0$, we can write (with $\Gamma$ as above)
\begin{equation}\label{e:Pqt-s7}
\mathcal P_{t}=\frac{1}{2\pi i}
\int_\Gamma \lambda^{k-1}(\lambda-\mathcal H_{t})^{-k}d\lambda. 
\end{equation}

Now we can proceed as in \cite{Kor18,ma-ma08}. We only observe 
that all the constants in the estimates in \cite{Kor18,ma-ma08} depend 
on finitely many derivatives of $g$, $h^L$, $\nabla^L$, $h^E$, 
$\nabla^E$ and the lower bound of $g$. Therefore, by
the bounded geometry assumptions, all the estimates are uniform on
the parameter $x_0\in X$. We will omit the details and give only the
final result.

 \begin{thm}\label{t:thm7.2}
The function $\mathcal P_{t}(Z,Z^\prime)$ admits an asymptotic expansion
\[
\mathcal P_{t}(Z,Z^\prime)\cong \sum_{r=0}^{\infty}F_{r}(Z,Z^\prime)t^r, \quad t\to 0.
\]
For any $j\in \mathbb N$, the remainder $\mathcal R_{j,t}(Z,Z^\prime):=\mathcal P_{t}(Z,Z^\prime)
-\sum_{r=0}^jF_{r}(Z,Z^\prime)t^r$ satisfies the condition: for any $m,m^\prime\in \mathbb N$, there exist $C>0$ and $M>0$ such that for any $0\leq t\leq 1$ and $Z,Z^\prime\in T_{x_0}X$ 
\begin{multline}
\sup_{|\alpha|+|\alpha^\prime|\leq m}\Bigg|\frac{\partial^{|\alpha|+|\alpha^\prime|}}{\partial Z^\alpha\partial Z^{\prime\alpha^\prime}}\mathcal R_{j,t}(Z,Z^\prime)\Bigg|_{C^{m^\prime}(X)}\\ \leq Ct^{j+1}(1+|Z|+|Z^\prime|)^M\exp(-c|Z-Z^\prime|). 
\end{multline}
\end{thm}

By \eqref{scaling}, we have
\[
P^0_{p,I}(Z,Z^\prime)=t^{-2n}\kappa^{-\frac 12}(Z)\mathcal P_{t}(Z/t,Z^\prime/t)\kappa^{-\frac 12}(Z^\prime), \quad Z,Z^\prime \in \mathbb R^{2n},
\]
that completes the proof of the asymptotic expansion \eqref{e:main-exp} in Theorem~\ref{t:main}.

\subsection{Computation of the coefficients}
We will use the formal power series technique developed in \cite[Section 1.5]{ma-ma08}. We take the formal asymptotic expansion for the operator $\mathcal H_{t}$ given by \eqref{e:Ht-formal}
and find a formal asymptotic expansion for the resolvent $(\lambda-\mathcal H_t)^{-1}$, $\lambda\in \Pi$, solving the formal power series equation 
\begin{equation}\label{e:Ltres-formal}
(\lambda-\mathcal H_t)f(t,\lambda)=I,
\end{equation}
where 
\[
f(t,\lambda)=\sum_{r=0}^{\infty} t^rf_r(\lambda), \quad f_r(\lambda)\in \operatorname{End}(L^2(\RR^{2n}, E_{x_0})).
\]
Identifying the coefficients in $t$, we get
\[
f_0(\lambda)=(\lambda-\mathcal H^{(0)})^{-1},
\] 
\[
f_r(\lambda)=(\lambda-\mathcal H^{(0)})^{-1}\sum_{j=1}^r\mathcal H^{(j)}f_{r-j}, \quad r\geq 1.
\] 
We find that 
\[
f_r(\lambda)=\sum_{\substack{k\geq 1, j_\ell \geq 1\\ j_1+\ldots+j_k=r}}(\lambda-\mathcal H^{(0)})^{-1}\mathcal H^{(j_1)}(\lambda-\mathcal H^{(0)})^{-1}\mathcal H^{(j_2)}\ldots \mathcal \mathcal H^{(j_k)} (\lambda-\mathcal H^{(0)})^{-1},
\]
Recall that $\mathcal P_{I}$ denotes the spectral projection of $\mathcal H^{(0)}$, corresponding to $I$, and put $\mathcal P^\bot_{I}=I-\mathcal P_{I}$. Using \eqref{e:Lambda-Bergman}, we can write
\[
(\lambda-\mathcal H^{(0)})^{-1}=\sum_{(\mathbf k,\mu) : \Lambda_{\mathbf k,\mu}\in I} \frac{1}{\lambda-\Lambda_{\mathbf k,\mu}}\mathcal P_{\Lambda_{\mathbf k,\mu}}+(\lambda-\mathcal H^{(0)})^{-1}\mathcal P^\bot_{I},
\]
Observe that the second term in the right hand side of this equality is an analytic function for $\lambda\in\Pi$. We infer that 
\begin{equation}\label{e:fr}
f_r(\lambda)=\Phi_{r}(\lambda)+\Phi^\bot_{r}(\lambda).
\end{equation}
where $\Phi^\bot_{r}$ is an analytic function in $\Pi$ given by
\[
\Phi^\bot_{r}(\lambda)=\sum_{\substack{k\geq 1, j_\ell \geq 1\\ j_1+\ldots+j_k=r}}(\lambda-\mathcal H^{(0)})^{-1}\mathcal P^\bot_{I} \mathcal H^{(j_1)}(\lambda-\mathcal H^{(0)})^{-1}\mathcal P^\bot_{I}\mathcal H^{(j_2)}\ldots \mathcal H^{(j_k)} (\lambda-\mathcal H^{(0)})^{-1}\mathcal P^\bot_{I},
\]
and 
\begin{equation}\label{e:Pr}
\Phi_{r}(\lambda)=\sum_{\substack{k\geq 1, j_\ell \geq 1\\ j_1+\ldots+j_k=r}}\mathcal R_0\mathcal H^{(j_1)}\mathcal R_1\mathcal H^{(j_2)}\ldots \mathcal H^{(j_k)} \mathcal R_k,
\end{equation}
where at least one of $\mathcal R_0, \ldots, \mathcal R_k$ equals $\frac{1}{\lambda-\Lambda_{\mathbf k,\mu}}\mathcal P_{\Lambda_{\mathbf k,\mu}}$ with $\Lambda_{\mathbf k, \mu}\in I$. Using \eqref{e:Pqt-s7} and \eqref{e:Ltres-formal}, we get a formal asymptotic expansion for $\mathcal P_{t}$:
\[
\mathcal P_{t}=\frac{1}{2\pi i}\int_\Gamma (\lambda-\mathcal H_{t})^{-1}d\lambda=\frac{1}{2\pi i} \sum_{r=0}^{\infty} t^r \int_\Gamma f_r(\lambda)d\lambda=\frac{1}{2\pi i} \sum_{r=0}^{\infty} t^r \int_\Gamma  \Phi_{r}(\lambda) d\lambda, 
\]
which gives 
\begin{equation}\label{e:FrPhi}
F_{r}=\frac{1}{2\pi i}\int_\Gamma \Phi_{r}(\lambda)d\lambda.
\end{equation}

For $r=0$, we have  
\[
\Phi_0(\lambda)=\sum_{(\mathbf k,\mu) : \Lambda_{\mathbf k, \mu}\in I} \frac{1}{\lambda-\Lambda_{\mathbf k,\mu}}\mathcal P_{\Lambda_{\mathbf k,\mu}}.
\]
By \eqref{e:FrPhi}, this proves \eqref{e:F0}.

Consider the set $\mathcal A$ of operators in $L^2(\RR^{2n},E_{x_0})$ with smooth kernel of the form $K(Z,Z^\prime)\mathcal P(Z,Z^\prime)$, where $K(Z,Z^\prime)$ is a polynomial in $Z, Z^\prime$ (here $\mathcal P(Z,Z^\prime)$ is the Bergman kernel given by \eqref{e:Bergman}). Let us show that, for any $\lambda\not\in \Sigma\cap I$, the operator $\Phi_{r}(\lambda)$ is in $\mathcal A$. By \eqref{e:FrPhi}, this will immediately imply that $F_r\in\mathcal A$ and prove \eqref{e:Fr}

It is easy to see that $\mathcal A$ is an involutive algebra with respect to the composition and the adjoint. Moreover, it is a filtered algebra with filtration given by the degree of the polynomial $K$.  By these properties, it is easy to see that it is suffices to prove that, for any $j_1,j_2,\ldots,j_k$, 
\begin{equation}\label{e:inclusion}
(\lambda-\mathcal H^{(0)})^{-1}\mathcal P^\bot_{I} \mathcal H^{(j_1)}\ldots (\lambda-\mathcal H^{(0)})^{-1}\mathcal P^\bot_{I} \mathcal H^{(j_k)}\mathcal P_{\Lambda_{\mathbf k,\mu}}\in \mathcal A_N
\end{equation}
with $N=\kappa(I)+j_1+\ldots+j_k+2k$.  

First, observe that $\mathcal P_{\Lambda_{\mathbf k,\mu}}\in \mathcal A_{\kappa(I)}$. Using the description of the coefficients $\mathcal H^{(j)}$ given by \eqref{e:Hj}, one can easily see that, for any $j$ and $A\in \mathcal A_N$, the operator $\mathcal H^{(j)}A$ belongs to $\mathcal A_{j+N+2}$. Finally, one can show that, for any $A\in \mathcal A_N$, the operator $(\lambda-\mathcal H^{(0)})^{-1}\mathcal P^\bot_{I} A$ belongs to $\mathcal A_N$. It follows immediately if we diagonalize the operator $\mathcal H^{(0)}$ in the orhonormal base of its eigenfunction and use the explicit description of its eigenvalues given, for instance, in \cite[Section 1.4]{ma-ma08}. This immediately completes the proof of \eqref{e:inclusion}.

\end{document}